\documentclass[12pt,reqno]{amsart}

\usepackage[margin=0.7in]{geometry}
\usepackage{amsmath}
\usepackage{amsfonts}
\usepackage{amssymb}
\usepackage{amsthm}
\usepackage{graphicx,caption,tikz}
\usepackage{multicol}
\usepackage[vcentermath,enableskew]{youngtab}
\usepackage{longtable}
\usepackage{hyperref}

\numberwithin{figure}{section}

\newcounter{theorem}
\newtheorem{thm}[theorem]{Theorem}
\newtheorem{conj}[theorem]{Conjecture}

\newtheorem{prop}[theorem]{Proposition}
\newtheorem{lem}[theorem]{Lemma}

\theoremstyle{definition}
\newtheorem{defn}[theorem]{Definition}
\newtheorem{expl}[theorem]{Example}

\numberwithin{theorem}{section}
\numberwithin{equation}{section}
\numberwithin{table}{section}

\newcommand{\al}{\alpha}
\newcommand{\la}{\lambda}
\newcommand{\ka}{\kappa}
\newcommand{\ep}{\epsilon}
\newcommand{\ze}{\zeta}
\newcommand{\be}{\beta}
\newcommand{\si}{\sigma}
\newcommand{\om}{\omega}

\newcommand{\HC}{\mathcal{H}}
\newcommand{\Q}{\mathbb{Q}}
\newcommand{\Pn}{\mathcal{P}}
\newcommand{\fd}{\mathfrak{d}}

\newcommand{\clmn}{c^{\la}_{\mu,\nu}}
\newcommand{\glmn}{g^{\la}_{\mu,\nu}}
\newcommand{\dlmn}{d^{\la}_{\mu,\nu}}
\newcommand{\lmn}{\la,\mu,\nu}
\newcommand{\set}[1]{\left\{#1\right\}}
\newcommand{\ab}[1]{\left\langle#1\right\rangle}

 \hypersetup
{
   colorlinks,        %
   citecolor=blue,%
   filecolor=black,%
   linkcolor=blue,%
   urlcolor=gray
}

\title[Jack and Macdonald Polynomials]{A product formula for certain Littlewood-Richardson coefficients for Jack and Macdonald polynomials}
\author{Yusra Naqvi}
\address{Department of Mathematics, Rutgers University, Piscataway, NJ 08854, USA}
\email{ynaqvi@math.rutgers.edu}
\date{}

\begin{document}

\begin{abstract}
Jack polynomials generalize several classical families of symmetric polynomials, including Schur polynomials, and are further generalized by Macdonald polynomials. In 1989, Richard Stanley conjectured that if the Littlewood-Richardson coefficient for a triple of Schur polynomials is 1, then the corresponding coefficient for Jack polynomials can be expressed as a product of weighted hooks of the Young diagrams associated to the partitions indexing the coefficient. We prove a special case of this conjecture in which the partitions indexing the Littlewood-Richardson coefficient have at most 3 parts. We also show that this result extends to Macdonald polynomials.
\end{abstract}

\maketitle
\section*{Introduction} 
\emph{Jack polynomials} $J_\la(\al;x)$ are a one parameter family of symmetric functions indexed by an integer partition $\lambda$. They were first introduced by Henry Jack \cite{J} in 1969 as generalizations of spherical functions over GL$(n,\mathbb{F})$/U$(n,\mathbb{F})$, where $\al=1/2,1,2$ correspond to the cases of $\mathbb{F}=\mathbb{H},\mathbb{C},\mathbb{R}$. Jack polynomials can be characterized in several ways. They appear as simultaneous eigenfunctions of certain Laplace-Beltrami type differential operators \cite{M}. In addition, they form an orthogonal basis for the ring of symmetric functions over the field of rational functions in $\al$. Jack polynomials were further generalized in 1988 by Macdonald polynomials $J_\la(q,t;x)$ \cite{M2}, which are a two parameter family of polynomials that reduce to Jack polynomials under a special limit. 

The $\al=1$ specialization gives us scalar multiples of the well-known Schur polynomials \cite{Ja,Sc}, which play a central role in the representation theory of the symmetric group $S_n$ as well as that of GL$(n,\mathbb{C})$. These polynomials are also indexed by partitions, and can be described combinatorially in terms of Young tableaux. Moreover, the coefficients that arise when a product of two Schur functions is decomposed into a sum of Schur functions have a combinatorial description known as the Littlewood-Richardson Rule (see \cite{M, HL}), given by counting the number of skew tableaux of a certain type. These Littlewood-Richardson coefficients also appear in various other fields outside of representation theory, such as in the study of Grassmanians and sums of Hermitian matrices (see \cite{HL, F2}). 

It is a continuing area of interest to find appropriate generalizations of these results for Schur polynomials in the context of Jack and Macdonald polynomials. Various works \cite{S,M,KS2,HHL,RY} establish several combinatorial properties of these polynomials and conjecture others. It is also possible to compute Littlewood-Richardson coefficients for such polynomials (see \cite{Sa2,Sa,Sch,Y}), but currently there are no formulas for these coefficients in the style of the Littlewood-Richardson rule. 

In this work, we prove a special case of one of Richard Stanley's conjectures \cite[Conj. 8.5]{S} which proposes a combinatorial description for certain Littlewood-Richardson coefficients for Jack polynomials in terms of a choice of upper and lower hooks (see Section \ref{ssec:stan}). In particular, this conjecture directly generalizes the Littlewood-Richardson rule for triples of partitions $(\lmn)$ such that the corresponding coefficient for Schur polynomials indexed by this triple is $1$. Moreover, this conjecture implies that the coefficient (under a certain explicit normalization) is a polynomial in $\al$ that can be written as a product of linear factors with positive integer coefficients. However, one of the main difficulties in proving this conjecture is that although it asserts that it is possible to write the coefficients as a product of some upper and lower hooks, it is not known how to make an appropriate choice of hooks. Also, while previous results, such as those in \cite{Y}, already present useful combinatorial descriptions for these coefficients, there are currently no formulas that prove Stanley's conjecture or even show the positivity of these coefficients.

Here we prove that this conjecture is true when the partitions in the triple $(\lmn)$ are restricted to having at most $3$ parts (Theorem \ref{thm:main}), and we extend this result to coefficients for Macdonald polynomials as well (Theorem \ref{thm:macmain}). We also show that the hooks can be chosen such that they preserve a convenient additional constraint which allows us to encode the coefficients much more simply in terms of a system of numbers we call \emph{division numbers} (defined in Section \ref{ssec:div}). In order to prove these assertions, we first divide the problem into several cases, and then present experimentally obtained formulas in terms of division numbers for the coefficients in each case of our classification. It turns out that in each of these cases, the verification of the formula uses one of two main lemmas (Lemmas \ref{lem:col} and \ref{lem:row}), giving us a unifying underlying structure.

In Section \ref{sec:prelim}, we provide some background about the combinatorics of partitions and symmetric functions. In Section \ref{sec:stan}, we give a precise statement of Stanley's conjecture for Littlewood-Richardson coefficients of Jack polynomials and state our main theorem. We classify all the partitions that satisfy the hypotheses of Stanley's conjecture in Section \ref{sec:class}. Then, in Section \ref{sec:proof}, we present the division number formulas, main lemmas, and a proof of the main theorem. In Section \ref{sec:macd}, we extend our result from coefficients for Jack polynomials to coefficients for Macdonald polynomials. Finally, we describe some ongoing work and further directions relating to our results in Section \ref{sec:future}.

\section{Preliminaries} \label{sec:prelim}
In this section, we present some basic definitions and background information pertaining to the theory of partitions and symmetric functions. We refer the reader to \cite{M,F2} for a more detailed treatment of this material.

\subsection{Partitions}

\begin{defn} A \emph{partition} $\la$ is a sequence $(\la_1,\la_2,\ldots,\la_n)$ of non-negative integers listed in weakly decreasing order: $$\la_1 \geq \la_2 \geq \cdots \geq \la_n \geq 0.$$
\end{defn}

Each nonzero $\la_i$ is called a \emph{part} of $\la$. We will sometimes write a partition $\la$ in the form $(i_1^{m_1},i_2^{m_2}, \ldots,i_k^{m_k})$, where $i_j^{m_j}$ denotes $m_j$ parts equal to $i_j$. We call $m_j$ the \emph{multiplicity} of $i_j$ in $\la$.

The \emph{length} $\ell(\la)$ of a partition $\la$ is the number of parts of $\la$. Let $\Pn_n$ denote the set of partitions of length at most $n$. We think of $\la \in \Pn_n$ as an $n$-tuple, with $\la_i=0$ for $i>\ell(\la)$.

The \emph{weight} $|\la|$ of $\la$ is the sum of its parts:
$$|\la|=\la_1+\la_2+ \cdots + \la_n.$$ 
If $|\la|=n$, then we say $\la$ is a \emph{partition of $n$}.

Given any two partitions $\la$ and $\mu$, we can define $\la+\mu$ as the partition obtained by taking the sum of $\la$ and $\mu$ as sequences:
$$(\la+\mu)_i=\la_i+\mu_i.$$

Given two partitions $\la,\mu$ of $n$, we say $\mu \leq \la$ if for all $i \in \set{1,\ldots,n}$, 
$$\mu_1 + \ldots +\mu_i \leq \la_1 + \ldots +\la_i.$$
The relation $\leq$ defines a partial order, known as the \emph{dominance order}, on the set of all partitions of $n$. 

Partitions are commonly represented diagramatically.
\begin{defn} The \emph{Young diagram} of a partion $\la$ is a left justified array of boxes such that there are $\la_i$ boxes in row $i$. 
(We will use the same symbol $\la$ to denote both the partition and its Young diagram.) 
\end{defn}

\begin{expl} Let $\la=(5,2,2,1).$ Then the corresponding Young diagram is:
$$\yng(5,2,2,1)$$
\label{ex:young}
\end{expl}

The \emph{conjugate} $\la'$ of a partition $\la$ is the partition whose diagram is the transpose of the diagram of $\la$, where the transpose is obtained by reflecting across the main diagonal and thus interchanging rows and columns.

\begin{expl} If $\la=(5,2,2,1)$ (as in Example \ref{ex:young}), then the transpose of its Young diagram is:
$$\yng(4,3,1,1,1)$$
and so $\la'=(4,3,1,1,1)$. 
\end{expl}

We say $\la \supset \mu$ if the diagram of $\la$ contains the diagram of $\mu$. Let $\la - \mu$ be the set theoretic difference between the two diagrams, which we call a \emph{skew diagram}.

\begin{expl} If $\la=(5,2,2,1)$ and $\mu=(3,2,1)$, then the skew diagram $\la/\mu$ is denoted by the marked boxes in the diagram below:
$$\young(~~~\bullet\bullet,~~,~\bullet,\bullet).$$ 
\label{e:skewd}
\end{expl}

If the skew diagram consists of $r=|\la|-|\mu|$ boxes and has at most one box in each column (respectively, row), we refer to it as a \emph{horizontal $r$-strip} (respectively \emph{vertical $r$-strip}). In Example \ref{e:skewd}, $\la/\mu$ is a horizontal $4$-strip. However, it is not a vertical strip since the first row of the skew diagram contains two boxes.

A \emph{skew tableau} $T$ is obtained by filling each box of a skew diagram $\la/\mu$ with a positive number, where $\la-\mu$ is called the \emph{shape} of $T$. If $m_i$ denotes the number of times $i$ appears in the skew tableau, we say $(m_1,\ldots, m_r)$ is the \emph{weight} of the $T$, and the \emph{word} $w(T)$ of $T$ is the sequence obtained by reading the entries of $T$ from right to left in each row.

\begin{expl} Let $T$ be the skew tableau given by
$$\young(~~~12,~1133,12,3)$$
Then:
\begin{itemize}
\item the shape of $T$ is $(5,5,2,1) - (3,1)$.
\item the weight of $T$ is $(4,2,3)$.
\item the word of $T$ is $w(T)=(2,1,3,3,1,1,2,1,3)$.
\end{itemize}
\end{expl}

A skew tableau $T$ is said to be \emph{semistandard} if the entries of $T$ weakly increase across rows (from left to right) and strongly increase down columns. We say that $T$ satisfies the \emph{Yamanouchi word condition} if the number of occurrences of an integer $i$ never exceeds the number of occurrences of $i-1$ for any initial segment of $w(T)$.

\begin{defn} A \emph{Littlewood-Richardson} tableau is a semistandard skew tableau $T$ that satisfies the Yamanouchi word condition.
\end{defn}

\begin{expl} The skew tableau
$$\young(~~~11,~112,23)$$
is a Littlewood-Richardson tableau.
\end{expl}

We will call any filling of a skew diagram that gives a Littlewood-Richardson tableau an \emph{LR filling}.

\subsection{Symmetric Functions}

Let $\mathbb{Z}[x_1,\ldots,x_n]$ denote the ring of polynomials in $n$ independent variables $x_1,\ldots,x_n$ with integer coefficients. Let $S_n$ be the symmetric group on $n$ letters. Then $S_n$ acts on $\mathbb{Z}[x_1,\ldots,x_n]$ by permuting the variables, and a polynomial is called \emph{symmetric} if it is unchanged under this action. The symmetric polynomials form a subring: $$\Lambda_n=\mathbb{Z}[x_1,\ldots,x_n]^{S_n}.$$

For each $\al=(\al_1,\ldots,\al_n) \in \mathbb{N}^{n}$ we can define the monomial $$x^\al = x^{\al_1}_1 \cdots x^{\al_n}_n.$$ Then we can define the \emph{monomial symmetric function} $m_\la$, where $\la$ is a partition of length at most $n$, by 
$$m_\la(x_1,\ldots,x_n) = \sum_{\al \in S_n\cdot\la} x^{\al},$$
where $S_n\cdot\la$ is the orbit of $\la$ under the action of $S_n$. The monomial symmetric functions form a $\mathbb{Z}$-basis for $\Lambda_n$.

For a partition $\la$, we can also define the skew-symmetric polynomial $a_\la$ by 
$$a_\la(x_1, \ldots, x_n) = \sum_{w \in S_n} \ep(w) x^{w(\la)},$$
where $\ep(w)$ is the sign of the permutation $w \in S_n$. Let $\delta$ be the partition $(n-1,n-2,\ldots,1,0)$. Then $a_{\la+\delta}$ is divisible by $a_\delta$, and the quotient 
$$s_\la(x_1,\ldots,x_n) = \frac{a_{\la+\delta}}{a_{\delta}},$$
called the \emph{Schur polynomial}, is a symmetric function. The $s_\la , \;\ell(\la) \leq n$ also form a basis for $\Lambda_n$.

Schur polynomials appear as spherical functions over GL$(n,\mathbb{C})/$U$(n,\mathbb{C})$. Spherical functions over GL$(n,\mathbb{F})/$U$(n,\mathbb{F})$ are further generalized by Jack polynomials $J_\la(\al;x_1,\ldots,x_n)$, where $\al=1/2,1,2$ correspond to the case of $\mathbb{F}=\mathbb{H},\mathbb{C},\mathbb{R}$, respectively.

To define Jack polynomials, we must first define the operator $D(\al)$ on $\Lambda \otimes \Q(\al)$ by
$$D(\al) = \frac{\al}{2} \sum_i x_i^2 \frac{\partial^2}{\partial x_i^2} + \sum_{i \neq j} \frac{x_i^2}{x_i-x_j} \frac{\partial}{\partial x_i}.$$
Then $D(\al)$ is upper triangular on the basis of monomial symmetric functions $m_\la$, ie
$$D(\al) m_\la = \sum_{\mu\leq \la} b_{\la,\mu} m_\mu.$$

\begin{defn} The \emph{monic} Jack polynomials $$P_\la = P_\la(\al;x_1, \ldots, x_n)=\sum_{\mu\leq \la} v_{\la,\mu} m_\mu$$ are the eigenfunctions of $D(\al)$ such that $v_{\la, \la} =1$. 
\end{defn}

Note that $P_\la(1)=s_\la$. We will also find it convenient to consider the following scalar multiples of $J_\la$:

\begin{defn} The \emph{integral} Jack polynomials $$J_\la=J_\la(\al;x_1, \ldots, x_n)=\sum_{\mu\leq \la} v_{\la,\mu} m_\mu$$ are the eigenfunctions of $D(\al)$ such that if $|\la|=m$, then $v_{\la, (1^m)}=m!$. 
\end{defn}

Jack polynomials are further generalized by Macdonald polynomials which are eigenfunctions of the operator $D(q,t)$ on $\Lambda \otimes \Q(q,t)$ defined by:
$$D(q,t) = \sum_i \left(\prod_{i \neq j} \frac{tx_i - x_j}{x_i - x_j} T_{q,i}\right),$$
where 
$$T_{q,i} f(x_1,\ldots,x_n) = f(x_1, \ldots, qx_i, \ldots, x_n).$$
Then, once again,
$$D(q,t) m_\la = \sum_{\mu\leq \la} b_{\la,\mu} m_\mu.$$

\begin{defn} The Macdonald polynomials $$P_\la = P_\la(q,t;x_1, \ldots, x_n)=\sum_{\mu\leq \la} v_{\la,\mu} m_\mu$$ are the eigenfunctions of $D(q,t)$ such that $v_{\la, \la} = 1$. 
\end{defn}

We can recover the Jack polynomials from the Macdonald polynomials by taking the limit as $q,t$ go to $1$, where the parameter $\al$ signifies the direction along which this limit is taken. Thus, $$\lim_{t\rightarrow 1} P_\la(t^\al,t)= P_\la(\al).$$

\subsection{The Littlewood-Richardson Rule}
Schur functions can be interpreted combinatorially, by the following theorem.

\begin{thm}
$$s_{\la} = \sum_{T} x^{\theta(T)},$$
where $T$ is a tableau of shape $\la$, and $\theta(T)$ is the weight of $T$.
\end{thm}

\begin{expl}
$s_{(2,1)} \in \Lambda_3:$
$$\young(11,2)\quad\young(11,3)\quad\young(12,2)\quad\young(12,3)\quad\young(13,2)\quad\young(13,3)
\quad\young(22,3)\quad\young(23,3)$$ 
$$s_{(2,1)} =x_1^2x_2 + x_1^2x_3 + x_1x_2^2 + 2x_1x_2x_3 + x_1x_3^2  + x_2^2x_3 + x_2x_3^2$$ 
\end{expl}

This also leads to a way of combinatorially interpreting the coefficients that appear when a product of Schur polynomials is expanded as a sum of Schur polynomials. This was developed using two major results. We start with a theorem that tells us how to expand such a product when one of the polynomials in the product is indexed by a partition of length 1. 

\begin{thm}[Pieri Rule] 
$$s_\mu s_{(r)} = \sum_\la s_\la,$$
where $\la/\mu$ is a horizontal $r$-strip.
\label{thm:spieri}
\end{thm}

\begin{expl} $\mu=(3,1), r=2$
$$\young(~~~11,~)\quad\young(~~~1,~1)\quad\young(~~~1,~,1)\quad\young(~~~,~11)\quad\young(~~~,~1,1)
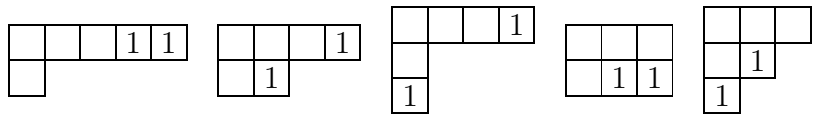
$$
$$s_{(3,1)} s_{(2)} = s_{(5,1)} +s_{(4,2)}+s_{(4,1,1)}+s_{(3,3)}+s_{(3,2,1)}$$
\end{expl}

Note that we can also consider the transpose of each of the indexing partitions, to get a way of multiplying two Schur polynomials when one of them is indexed by a partition consisting of a single column. 

Finally, we can extend this result to products of two Schur polynomials indexed by general partitions. This is done using the Littlewood-Richardson rule.

\begin{thm}[Littlewood-Richardson Rule] 
$$s_\mu s_\nu = \sum_\la c^{\la}_{\mu,\nu} s_\la,$$
where $c^{\la}_{\mu,\nu}$ is the number of Littlewood-Richardson tableaux $T$ of shape $\la/\mu$ and weight $\nu$.
\end{thm}

\begin{expl} $\mu=(2,1), \nu=(2,1), \la=(3,2,1)$

$$\young(~~1,~1,2)~\young(~~1,~2,1)~\young(~~2,~1,1)$$ 

\vspace{-4.5pc} \hspace{22pc} \begin{tikzpicture} \draw[thick] (0,0) --(1,2); \end{tikzpicture}

$$c^{(3,2,1)}_{(2,1),(2,1)} = 2$$
\end{expl}

\section{Stanley's Conjecture} \label{sec:stan}
\subsection{Statement of Conjecture} \label{ssec:stan}
We wish to generalize the Littlewood-Richardson rule to obtain a description of the coefficients that appear when a product of Jack or Macdonald polynomials is expanded as a sum of the respective polynomials. While it is possible to compute these coefficients recursively (see \cite{Sa2,Sa}), there is currently no combinatorial result that clearly reduces to the Littlewood-Richardson rule as we take the appropriate limit of the Jack or Macdonald polynomials to recover the corresponding Schur polynomials. However, in \cite{S}, Stanley made some observations and conjectures that give us some steps towards this goal. While Stanley discusses only the case of Jack polynomials in his paper, all results can be generalized to Macdonald polynomials as well.

In order to state Stanley's Conjecture \cite[Conj. 8.5]{S}, we must first define the hook length for a box in a Young diagram and some of its analogues.
The \emph{hook-length} $h_\la(b)$ of a box $b$ in the partition $\la$ is obtained by counting all the boxes to the right of $b$ (called the \emph{arm}, denoted $a_\la(b)$) and all the boxes below $b$ (called the \emph{leg}, denoted $\ell(b)$) along with $b$ itself.

\begin{align*}
a_\la(i,j) &= \la_i - j \\
\ell_\la(i,j) &= \la_j' - i \\
h_\la(i,j) &= a_\la(i,j) + \ell_\la(i,j) + 1
\end{align*}

\begin{expl}$\la=(5,2,2,1), b=(1,2)$
$$\young(~\times---,~|,~|,~)
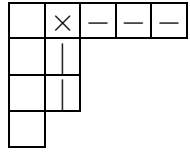
$$
$$h_\la (b) = 3+2+1 = 6$$
\end{expl}

We can define 2 $\al$-generalizations of $h_\la(b)$:
\begin{itemize}
\item \emph{upper hook-length}: $h^*_\la(b) = \al(a(b) +1)+\ell(b)$
\item \emph{lower hook-length}: $h_*^\la(b) = \al(a(b))+\ell(b)+1$
\end{itemize}
In effect, the upper hook treats the corner box as part of the arm, whereas the lower treats it as part of the leg. 

We also define the following products of hook lengths:
\begin{align*}
H^\la_* &= \prod_{b \in \la} h^\la_*(b) \\
H_\la^* &= \prod_{b \in \la} h_\la^*(b) \\
j_\la &= H^\la_* \cdot H_\la^* 
\end{align*}
Then we can relate the integral and the monic Jack polynomials as follows:
$$J_\la(\al) = H^\la_* P_\la(\al).$$ 

We can also define the dual $J^*_\la(\al)$ of $J_\la(\al)$ under the canonical inner product by:
$$J^*_\la(\al) = j_\la^{-1} J_\la(\al).$$ 

Finally, we consider the following expansions:
$$P_\mu P_\nu = \sum_\la c^{\la}_{\mu\nu}(\al) P_\la.$$
\begin{align*}
J_\mu J_\nu &= \sum_\la g^{\la}_{\mu\nu}(\al) J^*_\la, \\
P_\mu P_\nu &= \sum_\la c^{\la}_{\mu\nu}(\al) P_\la.
\end{align*} 
Then $$g^{\la}_{\mu\nu}(\al) = H^*_\la H_*^\mu H_*^\nu \clmn(\al).$$

We are now ready to state Stanley's conjecture.

\begin{conj}[Stanley, 1989] Given partitions $\lmn$ such that $\clmn(1)=1$, then for all $\al$,
\begin{equation}
g^{\la}_{\mu,\nu}(\al) = \left(\prod_{b \in \la} \tilde{h}_{\la}(b)\right) \left(\prod_{b \in \mu} \tilde{h}_{\mu}(b)\right) \left(\prod_{b \in \nu} \tilde{h}_{\nu}(b)\right),
\label{eq:stan}
\end{equation}
where each $\tilde{h}_\xi(b)$ is either $h^{*}_\xi(b)$ or $h^{\xi}_*(b)$. Moreover, we can choose these hooks such that there is an equal number of upper and lower hooks.
\label{conj:stan}
\end{conj}

Unfortunately, while this conjecture states that such a choice is always possible, there is no canonical way to make such a choice, and no conjecture for an assignment that might work in general. In fact, as Stanley himself notes in \cite{S}, there is often more than one assignment of upper and lower hooks that would satisfy this conjecture. In particular, he presents the following example, computed by Philip Hanlon.

\begin{expl} $\la=(2,2,2,1,1),\mu=(2,1,1),\nu=(2,1,1)$
$$\young(ll,ll,u?,l,?) \qquad  \young(u?,l,?) \qquad \young(u?,l,?)$$
Of the 6 boxes marked ``?", 5 must be taken to be upper hooks and 1 to be a lower hook, so there are 6 possible ways to obtain the correct coefficient.
\end{expl}

Since we can get $\clmn$ by dividing $\glmn$ by all the upper hooks in $\la$ and all the lower hooks in $\mu$ and $\nu$, we will call such hooks \emph{standard hooks} and boxes assigned to have standard hooks in Equation \ref{eq:stan} to be \emph{standard boxes}. On the other hand, we will call lower hooks in $\la$ and upper hooks in $\mu$ and $\nu$ \emph{flipped hooks} and boxes with such an assignment in Equation \ref{eq:stan} \emph{flipped boxes}. If $\glmn(\al)$ is given by a product of only standard hooks, then $\clmn(\al)=1$ for all $\al$. In general, $\clmn(\al)$ can be regarded as a product over flipped boxes of the ratio of the flipped hook to the standard hook. When $\al=1$, the upper and lower hooks have the same value, and so any such product reduces to 1, in agreement with the hypothesis $\clmn(1)=1$.   

We will call any triple $(\la,\mu,\nu)$ of partitions that satisfy the hypothesis $\clmn(1)=1$ a \emph{minimal triple}. Such triples correspond to the case of a unique Littlewood-Richardson tableau of shape $\la - \mu$ with weight $\nu$, but it remains difficult to generate all such triples in general. Minimal triples lie on the boundary of Horn cones, which are given by the eigenvalues of Hermitian matrices $A,B,C$ such that $A+B+C=0$. (However, note that not all boundary triples are minimal.) Minimal triples also play a prominent role in Fulton's conjecture, which states that a minimal triple remains minimal under a scaling of all three partitions by the same factor. (A proof of Fulton's conjecture is given by Knutson, Tao and Woodward  in \cite{KTW}.) 

\subsection{Main Theorem}

In this work, we prove the following special case of Stanley's conjecture.

\begin{thm} Stanley's conjecture is true for $\lmn \in \Pn_3$.
\label{thm:main}
\end{thm}

We will show this by first classifying all minimal triples of partitions in $\Pn_3$, which we do in Section \ref{sec:class}. We thus divide the problem into several cases and develop an experimental formula in the form of Equation \ref{eq:stan} for $\clmn$ in each case. A complete list of these is given in Section \ref{sec:div}. In Section \ref{sec:ver}, we verify that our experimental formulas indeed give the correct coefficient, thus completing the proof of Theorem \ref{thm:main}. In Section \ref{sec:macd}, we extend this theorem to get Theorem \ref{thm:macmain}, which shows that the coefficient for the corresponding Macdonald polynomials can also be obtained for minimal triples of partitions in $\Pn_3$ using the same system of upper and lower hook assignments using a suitable generalization of hook-lengths.

\subsection{The Pieri Rule for Jack Polynomials}

By Theorem \ref{thm:spieri}, we see that if $\nu$ consists of a single row (or column), $(\lmn)$ must be a minimal triple. In fact, we have an analogue of this theorem that gives a proof of Stanley's conjecture when $\nu$ falls into this special case.

\begin{thm}[Pieri Rule for columns {\cite[Thm 6.3]{KS1}}] If $\la/\mu$ is a vertical $r$-strip and $\nu=(1^r)$, then 
$$\clmn(\al) = \prod_{s \in X(\la/\mu)} \frac{h^{\la}_*(s)}{h_{\la}^*(s)} \frac{h^{*}_\mu(s)}{h_{*}^\mu(s)},$$
where $X(\la/\mu)$ denotes all the boxes $(i,j) \in \mu$ such that $\mu_i=\la_i$ and $\mu_j' < \la_j'$.
\label{thm:jpieri} 
\end{thm}

\begin{expl} $\la=(4,2,2), \mu=(3,2,1), \nu=(1,1)$ 

$$\young(uuuu,ul,uu) \qquad  \young(lll,lu,l) \qquad \young(l,l)$$

$$\clmn = \frac{2\al}{1+\al}$$
$$g^{\la}_{\mu,\nu} = 32\al^5(3+2\al)(1+2\al)^2(2+\al)^2(2+3\al)$$ 
\label{e:jpieri}
\end{expl}

We define $$b_\la(\al)=\frac{H^{\la}_*(\al)}{H^*_\la(\al)}.$$ Thus, we can think of $b_\la(\al)$ as an operator that switches upper and lower hooks. This gives us the following equation:

\begin{equation} \displaystyle c^{\la'}_{\mu',\nu'}\left(\frac{1}{\al}\right) = \frac{\clmn(\al) b_\mu(\al) b_\nu(\al)}{b_\la(\al)}.
\label{eq:transpose}
\end{equation}

Therefore, if $\lmn$ is a minimal triple and we transpose all 3 partitions, the resulting Littlewood-Richardson coefficient corresponds to swapping all the upper and lower hooks. This allows us to use the Pieri rule for columns as a rule for rows as well.

\begin{expl} We consider the triple obtained by transposing the partitions in Example \ref{e:jpieri}: \\ $\la=(3,3,1,1), \mu=(3,2,1), \nu=(2)$ 

$$\young(lll,lul,l,l) \qquad  \young(uuu,ul,u) \qquad \young(uu)$$

$$\clmn = \frac{16\al^2(1+2\al)}{3(1+\al)^4}$$
$$g^{\la}_{\mu,\nu} = 32\al^5(2+3\al)(1+2\al)^2(2+\al)^2(3+2\al)$$ 
\end{expl}

\section{Classification} \label{sec:class}
We present a classification of all minimal triples $(\lmn)$ consisting of partitions in $\Pn_3$. In particular, we show that such triples correspond to each face of co-dimension one of the $n=3$ Horn cone (see \cite{KTW}). It turns out that this correspondence is no longer true if we allow partitions of greater length, in which case minimal triples form a proper subset of the triples that lie on boundary faces of the associated Horn cone.

\subsection{Horn's Inequalities} \label{sec:horn}
Horn cones were defined by \cite{H} to answer the following problem: given two $n\times n$ Hermitian matrices $A$ and $B$ with eigenvalues $\mu=(\mu_1,\ldots,\mu_n)$ and $\nu=(\nu_1,\ldots,\nu_n)$ (arranged in weakly decreasing order), we wish to determine the possible eigenvalues $\la=(\la_1,\ldots,\la_n)$ of the sum $C=A+B$. Horn conjectured a list of inequalities involving $\la,\mu,\nu$ that, together with the condition $|\la| = |\mu| + |\nu|$, determine all possible combinations. These inequalities were verified by the works of Klyachko \cite{Kl} and of Knutson and Tao \cite{KT}, which also show that the Littlewood-Richardson coefficient $\clmn$ is nonzero if and only if $(\lmn)$ lie in the Horn cone $\mathcal{H}_n$. Later, Knutson, Tao and Woodward \cite{KTW} determined the minimal necessary list of such inequalities that determines this cone.

Using this list of inequalities for $\mathcal{H}_3$, we have that the Littlewood-Richardson coefficient $\clmn$ is nonzero if the partitions $\lmn \in \Pn_3$ are such that $|\la| = |\mu| + |\nu|$, and they satisfy all of the inequalities in Table \ref{tab:horn} below.

\begin{table}[h!]
\caption{Horn's Inequalities for $n=3$}
\begin{multicols}{3}
\begin{enumerate}
\item $\mu_3 \leq \mu_2$ 
\item $\mu_2 \leq \mu_1$
\item $\nu_3 \leq \nu_2 $
\item $\nu_2 \leq \nu_1$
\item $\la_3 \leq \la_2 $
\item $\la_2 \leq \la_1$
\item $\la_1 \leq \mu_1+\nu_1$
\item $\la_2 \leq \mu_1+\nu_2$
\item $\la_2 \leq \mu_2+\nu_1$
\item $\la_3 \leq \mu_1+\nu_3$
\item $\la_3 \leq \mu_2+\nu_2$
\item $\la_3 \leq \mu_3+\nu_1$
\item $\la_3 \geq \mu_3+\nu_3$
\item $\la_2 \geq \mu_3+\nu_2$
\item $\la_2 \geq \mu_2+\nu_3$
\item $\la_1 \geq \mu_3+\nu_1$
\item $\la_1 \geq \mu_2+\nu_2$
\item $\la_1 \geq \mu_1+\nu_3$
\end{enumerate}
\end{multicols}
\label{tab:horn}
\end{table}

It is known that minimal triples $(\lmn)$ all lie on a union of some faces of the Horn cone (see \cite{B,KTW}). We will refer to a face of codimension one as a \emph{facet}. Since each facet is obtained by changing one of the defining inequalities to an equality, for $\HC_3$, we will refer to each facet by the same number as the corresponding inequality as in Table \ref{tab:horn} above.

In general, not every facet of $\HC_n$ contains minimal triples. However, this does hold for $\HC_3$, and so one can check triples $(\lmn)$ on the interior of each face, and determine that every single facet does indeed give a minimal triple. In the next section, we present a direct combinatorial proof of this fact. 

\subsection{Littlewood-Richardson Tableaux}
We will show that each facet of $\HC_3$ contains minimal triples by classifying the possible Littlewood-Richardson tableaux of shape $\la/\mu$ of weight $\nu$ in the case that $\lmn \in \mathcal{P}_3$. The cases presented in this proof were also used to determine the experimentally obtained formulas for $\clmn(\al)$ presented in Section \ref{sec:div}. 

\begin{thm}
For partions $\lmn \in \Pn_3$, we have $\clmn(1) = 1$ if and only if $\lmn$ lie on a facet of the Horn cone $\mathcal{H}_3$.
\label{thm:horn}
\end{thm}

\begin{proof} A skew diagram of shape $\la/\mu$ consists of at most three rows. Therefore, if $\nu$ has length 3, then any LR filling of $\la/\mu$ of weight $\nu$ must consist of at least $\nu_3$ occurrences of $i$ in row $i$. We therefore only need to consider the remaining boxes, and we can thus assume, without loss of generality, that $\nu$ has length at most 2. By symmetry, we can also assume the same for $\mu$. 

Now let $T$ be a Littlewood-Richardson tableau of shape $\la/\mu$ with weight $\nu$. Then $w(T)$ must be a sequence of 1's and 2's of the form $(1^{a_1},2^{b_2},1^{b_1},2^{c_2},1^{c_1}),$ where $i^m$ denotes $m$ consecutive occurrences of $i$. In order to satisfy the Yamanouchi word condition, we must require that $a_1 \geq b_2$ and $a_1+b_1\geq b_2 + c_2$. For instance, if $T_1$ is the following diagram:
$$\young(~~~~~~111,~~~1122,1222),$$
then $w(T_1)=(1^3,2^2,1^2,2^3,1^1)$. Note, however, that in this case, a filling of this skew diagram of weight $(6,5)$ is not unique. We must therefore determine which restrictions on the set $({a_1},{b_2},{b_1},{c_2},{c_1})$ of multiplicities in $w(T)$ lead to a minimal triple $(\lmn)$.

First, suppose every column in $\la/\mu$ consists of a single box, so that $\la/\mu$ is a horizontal $|\nu|$-strip:
$$\young(~~~~1,~~12,12).$$
In order to have a unique LR filling, either $b_2=0$ (type B) or $c_1=0$ (type C). To see this, consider the case of an LR filling in which both $b_2$ and $c_1$ are nonzero, as in the diagram above. Then the last 1 in the third row can be swapped with the first 2 in the second row to get another LR filling,
$$\young(~~~~1,~~11,22)$$
so $\clmn(1)$ must be greater than 1 in this case. However, as this second diagram illustrates, requiring that the filling be of type B or C is not sufficient to give a minimal triple, even though it is a necessary condition. Specifically, in the absence of any additional restrictions, it may be possible to swap a 2 in the third row with a 1 in the second row.

Therefore, for each type, B or C, we require one of the following restrictions:
\begin{enumerate}
\item[I.] $c_2=0$ 
\item[II.] $b_1=0$ 
\item[III.] $a_1=b_2$ 
\item[IV.] $a_1+b_1=b_2+c_2$. 
\end{enumerate}
Conditions I and II remove one of the quantities that would have been involved in such a swap to get a new LR filling with the same weight. Conditions III and IV imply that any such swap would violate the Yamanouchi word condition, since the swap would have the effect of increasing $b_2$ while leaving $a_1,b_1$ and $c_2$ unchanged.

Finally, we consider the case in which $\la/\mu$ is no longer necessarily a horizontal strip. Then every column in the skew diagram could have up to two boxes, and whenever it does contain two boxes, the filling must be a 1 in the upper box and a 2 in the lower box. We could have an overlap between the first and second rows (denoted by type $o_1$) or an overlap between the second and third rows (denoted type $o_2$). In the case that $o_i$ does not occur, we denote the number of columns in the gap between the rows of the skew diagram by $g_i$. Thus, we have 32 cases in all (type B or C, type I-IV, type $o_1$ or $g_1$, and type $o_2$ or $g_2$). 

\begin{table}[h!]
\caption{Minimal triples of partitions in $\Pn_3$}
\begin{tabular}{|l||c|c|c|c|}
\hline
Type & $g_1g_2$ & $g_1o_2$ & $o_1g_2$ & $o_1o_2$ \\
\hline
\hline
B.I & $(3)$ & $(11)$ & $(8)$ & $(16)$ \\
\hline
B.II & $(15)$ & $(5)$ & $(2)$ & $(10)$ \\
\hline
B.III & $(18)$ & $(18)$ & $(6)$ & $(6)$ \\
\hline
B.IV & $(12)$ & $(12)$ & $(17)$ & $(17)$ \\
\hline
C.I & $(13)$ & $(1)$ & $(13)$ & $(1)$ \\
\hline
C.II & $(7)$ & $(14)$ & $(7)$ & $(14)$ \\
\hline
C.III & $(9)$ & $(9)$ & $(9)$ & $(9)$ \\
\hline
C.IV & $(4)$ & $(4)$ & $(4)$ & $(4)$ \\
\hline
\end{tabular}
\label{tab:class}
\end{table}

We will use $o_i$ and $g_i$ not only as a label for each type, but also a count (analogous to $a_1, b_i, c_i$) of the number of overlapping columns in the skew diagram, or the number of columns in the gap between rows of the skew diagram. Therefore, in general, the parts of $\lmn$ are given by:
\begin{align*}
\nu_1 &= a_1+b_1+c_1+o_1+o_2+\nu_3 \\
\nu_2 &= b_2+c_2+o_1+o_2+\nu_3 \\
\mu_1 &= b_1+b_2+g_1+o_2+c_1+c_2+g_2+\mu_3 \\
\mu_2 &= c_1+c_2+g_2+\mu_3 \\
\la_1 &= b_1+b_2+g_1+o_2+c_1+c_2+g_2+a_1+o_1+\mu_3+\nu_3 \\ 
\la_2 &= b_1+b_2+o_2+c_1+c_2+g_2+o_1+\mu_3+\nu_3 \\ 
\la_3 &= o_2+c_1+c_2+\mu_3+\nu_3 
\end{align*}
Therefore, each of the 32 cases corresponds to a restriction on the partitions $\lmn$. For instance, B.I.$g_1g_2$ means that $b_2=c_2=o_1=o_2=0$, and therefore $\nu_2=\nu_3$. Similarly, B.II.$g_1o_2$ means that $b_2=c_2=o_1=g_2=0$, and so we get that $\mu_2+\nu_2=\la_3$. We give a complete list of restrictions in Table \ref{tab:class}, where each number refers to the facet of $\HC_3$ determined by the correspondingly numbered Horn inequality above. We thus verify that each facet of $\HC_3$ appears in this table, and therefore each must contain only minimal triples.

\end{proof}

\section{Proof of Main Theorem} \label{sec:proof}
\subsection{Division Numbers} \label{ssec:div} Every partition $\la$ can be divided into rectangular \emph{blocks} consisting of all columns of the same height. We will use $\om^\la_i$ to denote the block $\left((\la_i-\la_{i+1})^i\right)$. Then if $\ell(\la)=n$, we can decompose $\la$ as the sum $\om^\la_1 + \om^\la_2 + \cdots +\om^\la_n$ of all its blocks.
\begin{expl} Let $\la=(6,4,2)$. The block $\om^\la_3$ is highlighted in the Young diagram below.
$$\young(\bullet\bullet~~~~,\bullet\bullet~~,\bullet\bullet)$$
\end{expl}

We refer to each part of a block $\om^\la_i$ as a \emph{strip}. Thus, each strip consists of a row within a block.
\begin{expl} Let $\la=(6,4,2)$. The strips $(\om^\la_3)_2$ and $(\om^\la_2)_1$ are highlighted in the Young diagram below.
$$\young(~~\bullet\bullet~~,\bullet\bullet~~,~~)$$
\end{expl}
 
It turns out that for a minimal triple $(\lmn)$ of partitions in $\Pn_3$, it is possible to obtain $\clmn$ by an assignment of upper and lower hooks in the corresponding diagrams such that within each strip, all the upper hooks that occur appear to the left of all the lower hooks that occur. (Note that a strip may contain only only upper hooks or only lower hooks.) We can thus encode the coefficient $\clmn$ by a system of \emph{division numbers}, which are numbers for each strip in $\lmn$ indicating the transition point between upper and lower hooks. By convention, we use the division numbers to count the flipped hooks in each strip, ie the lower hooks in each strip of $\la$ and the upper hooks in each strip of $\mu$ and $\nu$.

For each partition, we write the division numbers in a matrix style array, arranged in the same order (left to right, top to bottom) as the strip to which they correspond. Note that the division number symbols differ from matrices in that they contain no entries below the off-diagonal. Moreover, as we prove in Lemma \ref{lem:3to2} below, all hooks in the blocks $\om^\mu_3$ and $\om^\nu_3$ can be taken to be lower hooks, corresponding to division numbers of $0$ for all the strips in those blocks. Therefore, we will write the division numbers for $\la$ within a $3\times 3$ array and those for $\mu$ and $\nu$ within a $2\times 2$ array.

%
%

\begin{expl} $\la=(8,7,4),\mu=(6,3),\nu=(5,5)$ 

$$\young(ulllullu,uullull,ulll) \qquad  \young(uuluuu,uul) \qquad \young(ullll,uuuul)$$

$\clmn$ is encoded by the division numbers given as follows.
$$\la: \begin{bmatrix}3 & 2 & 0 \\ 2 & 2 \\ 3\end{bmatrix}
\quad \mu: \begin{bmatrix}2 & 3\\2 \end{bmatrix}
\quad \nu: \begin{bmatrix}1 & 0\\4 \end{bmatrix}.$$
\end{expl}

\subsection{Algebraic Structures} To compute $\clmn$ from the division numbers, we require the following notation. 

Let $\be$ be a multiset. We regard $\be$ as the set of vanishing points (counted with multiplicity) of a polynomial. Therefore, let $\phi(x;\be)$ be the smallest degree polynomial in $x$ such that $\phi(b;\be)=0$ for all nonzero $b \in \be$ and $\phi(0;\be)=1$. In particular, we have:
$$\phi(x;\be) := \prod_{b \in \be, b \neq 0}\left(\frac{b-x}{b}\right).$$

Such polynomials give us a natural way to write the coefficients $\clmn(\al)$ for minimal triples.
Given an upper hook $h^*_\la$, we can write the ratio of the corresponding lower hook to the upper hook as
$$\frac{h^*_\la-(\al-1)}{h^*_\la}.$$ Also, given a lower hook $h^\mu_*$, we can write the ratio of the corresponding upper hook to the lower hook as $$\frac{-h^\mu_*-(\al-1)}{-h^\mu_*}.$$ Thus, each $\clmn(\al)$ can we written as $\phi(\al-1;\mathcal{F}(\lmn))$, where $\mathcal{F}(\lmn)$ is the set of standards hooks of flipped boxes in $\la$ and negatives of standard hooks of flipped boxes in $\mu$ and $\nu$.

We will also find it convenient to write our hooks in terms of $r=1/\al$. In this case, we can regard our hook-lengths as
\begin{align*}
h^*_\la(b) &= a(b)+1 +\ell(b)r \\
h_*^\la(b) &= a(b) + (\ell(b)+1)r 
\end{align*}
and we have that
$$\phi(\al-1;m\al+n) = \frac{(m-1)\al+n+1}{m\al+n} = \frac{(m-1)+(n+1)r}{m+nr} = \phi(1-r;m+nr).$$

Succesive flipped $r$-hooks within a single strip differ by $1$, and so we require an effective way to describe such products. To do this, we will first define the following notation: 
$$\ab{x;a}_j = \phi(x;\set{a,\ldots,a+j-1}).$$
When $x$ is fixed and clear from context, we will suppress it and simply write $\ab{a}_j$. 

We will make use of two main identities involving such terms. 

For the first identity, observe that if $j=j_1+j_2$, then
\begin{align*}
\ab{a}_{j_1} \ab{a+j_1}_{j_2} = \ab{a}_j = \ab{a}_{j_2} \ab{a+j_2}_{j_1}
\end{align*}
and so 
\begin{equation}
\frac{\ab{a}_{j_1}}{\ab{a+j_2}_{j_1}} = \frac{\ab{a}_{j_2}}{\ab{a+j_1}_{j_2}}.
\label{eq:phid1} 
\end{equation}

For the second identity, note that if $a+b=x$ then
$$\left(\frac{a-x}{a}\right)\left(\frac{b-x}{b}\right) = \left(\frac{a-x}{a}\right)\left(\frac{-a}{x-a}\right) =1 $$
and more generally that
\begin{equation}
\ab{a}_{j} \ab{b-j+1}_j = 1,
\label{eq:phid2}
\end{equation}
where the $i^{\mbox{th}}$ term in the first product cancels with the $(j-i+1)^{\mbox{th}}$ term in the second product, since $a+(i-1)+(b-j+1)+(j-i)=a+b=x$. Note that this is equivalent to saying that 
$\ab{a}_{j} \ab{b}_j = 1$ whenever $a+b=x-j+1$.

Now note that such terms can be used to describe the product of flipped hooks within a single strip. We will use the notation:
$$[b;n] =\ab{1-r;b+1}_n = \phi\left(1-r;\set{b+1,b+2,\ldots, b+n}\right).$$
Let $h^{ij}_\la$ denote $h^*_\la(i,1)-h^*_\la(j,1)$. Then given partitions $\lmn$ and a set of division numbers $\mathfrak{n}$ for each strip in these partitions, we define $\mathbf{d}^\la_{\mu,\nu}(\mathfrak{n})$ to be the product: $$\mathbf{d}^\la_{\mu,\nu}(\mathfrak{n}) = \prod_{i\leq j} [h_\la^{ij};n^\la_{ij}]\cdot [-h_\mu^{ij};n^\mu_{ij}]\cdot [-h_\nu^{ij};n^\nu_{ij}],$$
where $n^\la_{ij}$ is the division number corresponding to the $i^{\mbox{\tiny th}}$ strip in $\om^{\la}_j$, and $n^\xi_{ij}$ is the division number corresponding to the $i^{\mbox{\tiny th}}$ strip in $\om^\xi_{j-1}$ for $\xi \in \set{\mu,\nu}$. We will refer to the starting point $b=\pm{h_\xi^{ij}}$ in each term of the form $[b;n]$ as the \emph{anchor} for the corresponding strip. 

Using equation \ref{eq:phid1}, we can determine how changes to the anchors or division numbers affect $\mathbf{d}^\la_{\mu,\nu}(\mathfrak{n})$. In particular, we have that 
\begin{align}
\frac{[h^{ij}_\xi;n^\xi_{ij} - t]}{[h^{ij}_\xi;n^\xi_{ij}]} &= \frac{1}{[h^{ij}_\xi+n^\xi_{ij} - t;t]}, \label{eq:mod1} \\
\frac{[h^{ij}_\xi+t;n^\xi_{ij}]}{[h^{ij}_\xi;n^\xi_{ij}]} &= \frac{[h^{ij}_\xi+n^\xi_{ij};t]}{[h^{ij}_\xi;t]}.
\label{eq:mod2}
\end{align}

\subsection{Division Numbers for Minimal Triples in $\Pn_3$} \label{sec:div} 
Let $\fd_{ijk}$ encode the quantity $|\la_i + \mu_j - \nu_k|$, and let $\xi_{ij}$ denote $\xi_i-\xi_j$ for any partition $\xi$. Let $\mathfrak{p}$ be the positive part of $\la_3-\mu_2-\nu_3$, so that $\mathfrak{p} = o_2 = \max(\la_3-\mu_2-\nu_3,0)$. Finally, let $x^{\pm}=x\pm \mathfrak{p},$ where $x$ is either some $\fd_{ijk}$ or some $\xi_{ij}$.

We present a complete list of division number formulas below for minimal triples in $\Pn_3$. These formulas are grouped according to facets of the Horn cone defined by the inequalities in Table \ref{tab:horn}. For each case, we present the division numbers for $\lmn$, and list the proposition in which this formula is verified. These propositions all appear in Section \ref{sec:ver}.

\begin{longtable}[h!]{|l|l|ccc|l|} 
\caption{Division numbers for minimal triples in $\Pn_3$ \label{tab:div}} \\
\hline
&Case: & $\la$: & $\mu$: & $\nu$: & Prop. \\
\hline
1.&$\mu_3=\mu_2$ & $\begin{bmatrix}
\fd_{333} & \fd_{222} & 0 \\
\fd_{333} & \fd_{222} & \\
\fd_{333} &&
\end{bmatrix}$ & $\begin{bmatrix}
0 & \mathfrak{d}_{111}\\
0 & 
\end{bmatrix}$ & $\begin{bmatrix}
\fd_{333} & \fd_{222} \\
\fd_{333} & 
\end{bmatrix}$ & \ref{prop:horn1} \\
&&&&&\\
2.& $\mu_2=\mu_1$ & $\begin{bmatrix}
\fd_{111} & \fd_{213} & 0 \\
\fd_{222} & \fd_{213} & \\
\fd_{111} &&
\end{bmatrix}$ & $\begin{bmatrix}
0 & 0\\
\fd_{333} & 
\end{bmatrix}$ & $\begin{bmatrix}
\fd_{213} & \fd_{111} \\
\fd_{213} & 
\end{bmatrix}$ & \ref{prop:horn2} \\
&&&&&\\
3.& $\nu_3=\nu_2$ & $\begin{bmatrix}
\fd_{333} & \fd_{222} & 0 \\
\fd_{333} & \fd_{222} & \\
\fd_{333} &&
\end{bmatrix}$ & $\begin{bmatrix}
\fd_{333} & \fd_{222} \\
\fd_{333} & 
\end{bmatrix}$ & $\begin{bmatrix}
0 & \mathfrak{d}_{111}\\
0 & 
\end{bmatrix}$ & \ref{prop:horn1} \\
&&&&&\\
4.& $\nu_2=\nu_1$ & $\begin{bmatrix}
\fd_{111} & \fd_{231} & 0 \\
\fd_{222} & \fd_{231} & \\
\fd_{111} &&
\end{bmatrix}$ & $\begin{bmatrix}
\fd_{231} & \fd_{111} \\
\fd_{231} & 
\end{bmatrix}$ & $\begin{bmatrix}
0 & 0\\
\fd_{333} & 
\end{bmatrix}$ & \ref{prop:horn2} \\
&&&&&\\
5.& $\la_3=\la_2$ & $\begin{bmatrix}
\fd_{111} & 0 & 0 \\
\fd_{333} & 0 & \\
\fd_{223} &&
\end{bmatrix} $ & $\begin{bmatrix}
\fd_{232} & \fd_{223}\\
\fd_{222} & 
\end{bmatrix}$ & $\begin{bmatrix}
\fd_{223} & \fd_{232} \\
\fd_{223} & 
\end{bmatrix}$ &
\ref{prop:la23} \\
&&&&&\\
6.& $\la_2=\la_1$ & $\begin{bmatrix}
\fd_{333} & \fd^-_{322} & 0 \\
\mathfrak{p} & \fd_{231} & \\
\fd_{221} &&
\end{bmatrix}$ & $\begin{bmatrix}
\fd^-_{212} & \mu_{12}\\
\fd_{221} & 
\end{bmatrix}$ & $\begin{bmatrix}
\fd^+_{231} & \fd_{221} \\
\fd_{231} & 
\end{bmatrix}$ & 
\ref{prop:la21} \\
&&&&&\\
7.& $\la_1 = \mu_1+\nu_1$
& $\begin{bmatrix}
0 & 0 & 0 \\
\fd_{333} & 0 & \\
\fd_{333} &&
\end{bmatrix}$ & $\begin{bmatrix}
0 & 0\\
\fd_{333} & 
\end{bmatrix}$ & $\begin{bmatrix}
0 & 0 \\
\fd_{333} & 
\end{bmatrix}$ & 
\ref{prop:nu1} \\
&&&&&\\
8.& $\la_2 = \mu_1+\nu_2$ 
& $\begin{bmatrix}
\fd_{333} & 0 & 0 \\
0 & 0 & \\
\fd_{333} &&
\end{bmatrix}$ & $\begin{bmatrix}
0 & 0 \\
\fd_{333} & 
\end{bmatrix}$ &
$\begin{bmatrix}
\fd_{333} & \fd_{111}\\
0 & 
\end{bmatrix}$ &  
\ref{prop:nu1} \\
&&&&&\\
9.& $\la_2 = \mu_2+\nu_1$ 
& $\begin{bmatrix}
\fd_{333} & 0 & 0 \\
0 & 0 & \\
\fd_{333} &&
\end{bmatrix}$ & $\begin{bmatrix}
\fd_{333} & \fd_{111}\\
0 & 
\end{bmatrix}$ & $\begin{bmatrix}
0 & 0 \\
\fd_{333} & 
\end{bmatrix}$ & 
\ref{prop:nu1} \\
&&&&&\\
10.& $\la_3 = \mu_1+\nu_3$ 
& $\begin{bmatrix}
\fd^+_{111} & \fd_{223} & 0 \\
\fd^+_{221} & \fd_{223} & \\
\mathfrak{p} &&
\end{bmatrix}$ & $\begin{bmatrix}
\mathfrak{p} & 0 \\
\fd^+_{223} & 
\end{bmatrix}$ & $\begin{bmatrix}
\fd_{111} & \fd^+_{223}\\
\fd_{221} & 
\end{bmatrix}$ &  
\ref{prop:nu1} \\
&&&&&\\
11.& $\la_3 = \mu_2+\nu_2$ 
& $\begin{bmatrix}
\fd_{333} & \fd_{222} & 0 \\
\fd_{333} & \fd_{222} & \\
0 &&
\end{bmatrix}$ & $\begin{bmatrix}
\fd_{332} & \fd_{223}\\
0 & 
\end{bmatrix}$ & $\begin{bmatrix}
\fd_{323} & \fd_{232} \\
0 & 
\end{bmatrix}$ & 
\ref{prop:la3} \\
&&&&&\\
12.& $\la_3 = \mu_3+\nu_1$ 
& $\begin{bmatrix}
\fd^+_{111} & \fd_{223} & 0 \\
\fd^+_{221} & \fd_{223} & \\
\mathfrak{p} &&
\end{bmatrix}$ & $\begin{bmatrix}
\fd_{111} & \fd^+_{223}\\
\fd_{221} & 
\end{bmatrix}$ & $\begin{bmatrix}
\mathfrak{p} & 0 \\
\fd^+_{223} & 
\end{bmatrix}$ & 
\ref{prop:nu1} \\
&&&&&\\
13.& $\la_3 = \mu_3+\nu_3$
& $\begin{bmatrix}
0 & \fd_{111} & 0 \\
0 & \fd_{111} & \\
0 &&
\end{bmatrix}$ & $\begin{bmatrix}
0 & \fd_{111}\\
0 & 
\end{bmatrix}$ & $\begin{bmatrix}
0 & \fd_{111} \\
0 & 
\end{bmatrix}$ & 
\ref{prop:nu3} \\
&&&&&\\
14.& $\la_2 = \mu_3+\nu_2$
& $\begin{bmatrix}
\fd_{323} & 0 & 0 \\
\fd_{333} & 0 & \\
\fd_{323} &&
\end{bmatrix}$ & $\begin{bmatrix}
\fd_{323} & 0\\
\fd_{333} & 
\end{bmatrix}$ & $\begin{bmatrix}
0 & \fd_{111} \\
0 & 
\end{bmatrix}$ & 
\ref{prop:nu3} \\
&&&&&\\
15.& $\la_2 = \mu_2+\nu_3$
& $\begin{bmatrix}
\fd_{332} & 0 & 0 \\
\fd_{333} & 0 & \\
\fd_{332} &&
\end{bmatrix}$ & $\begin{bmatrix}
0 & \fd_{111} \\
0 & 
\end{bmatrix}$ & $\begin{bmatrix}
\fd_{332} & 0\\
\fd_{333} & 
\end{bmatrix}$ & 
\ref{prop:nu3} \\
&&&&&\\
16.& $\la_1 = \mu_3+\nu_1$
& $\begin{bmatrix}
\fd_{333} & \fd_{232} & 0 \\
\fd_{323} & \fd_{222} & \\
\fd_{323} &&
\end{bmatrix}$ & $\begin{bmatrix}
0 & \mu_{12}\\
0 & 
\end{bmatrix}$ & $\begin{bmatrix}
\fd_{333} & \fd_{232} \\
\fd_{323} & 
\end{bmatrix}$ & 
\ref{prop:nu3} \\
&&&&&\\ 
17.& $\la_1 = \mu_2+\nu_2$ 
& $\begin{bmatrix}
\fd_{111} & \fd^-_{223} & 0 \\
\fd^+_{222} & \fd_{113} & \\
\fd^+_{121} &&
\end{bmatrix}$ & $\begin{bmatrix}
\fd^-_{112} & 0\\
\fd_{221} & 
\end{bmatrix}$ & $\begin{bmatrix}
\fd^+_{213} & \fd_{121} \\
\la^{+}_{23} & 
\end{bmatrix}$ &
\ref{prop:la1}\\
&&&&&\\
18.& $\la_1 = \mu_1+\nu_3$
& $\begin{bmatrix}
\fd_{333} & \fd_{223} & 0 \\
\fd_{332} & \fd_{222} & \\
\fd_{332} &&
\end{bmatrix}$ & $\begin{bmatrix}
\fd_{333} & \fd_{223} \\
\fd_{332} & 
\end{bmatrix}$ &
$\begin{bmatrix}
0 & \nu_{12}\\
0 & 
\end{bmatrix}$ &  
\ref{prop:nu3} \\
&&&&&\\
\hline 
\end{longtable}

Note that all division numbers that appear in Table \ref{tab:div} are positive and do not exceed the size of the strip in which they appear.

\subsection{Minimal Paths}
In order to verify the proposed formulas for the coefficient $\clmn$, we will typically use induction on $|\la|-|\mu|$. In order to do this, we decompose $\nu$ into two pieces $\nu'$ and $\nu''$, and compute the coefficients obtained when we expand the product $P_\mu P_{\nu'}P_{\nu''}$ as a sum. Using associativity, we can expand this product in 2 different ways. 

\begin{lem} For fixed $\la,\mu,\ze,\ep$, $$\sum_{\ka \subset \la} c^\ka_{\mu,\ze} \cdot c^{\la}_{\ka,\ep} = \sum_{\eta \subset \la} c^{\eta}_{\ze,\ep} \cdot c^\la_{\mu,\eta}.$$
\label{lem:path}
\end{lem}

\begin{proof} We use the associativity of product $P_\mu P_\ze P_\ep$ to expand the coefficient of $P_\la$ in this product as a sum in 2 different ways:
\begin{align*}
(P_\mu P_{\ze})P_\ep &= \left(\sum_\ka c^{\ka}_{\mu,\ze} P_\ka \right) P_{\ep} \\
&= \sum_\xi \sum_\ka c^{\ka}_{\mu,\ze} \cdot c^{\xi}_{\ka,\ep} \; P_\xi 
\end{align*}
\begin{align*}
P_\mu (P_{\ze}P_\ep) &= P_\mu \left(\sum_\eta c^{\eta}_{\ze,\ep} P_{\eta} \right)\\
&= \sum_\xi \sum_\eta c^{\eta}_{\ze,\ep} \cdot c^{\xi}_{\mu,\eta} \; P_\xi.
\end{align*}
Picking out the coefficient of $P_\la$ in this expression tells us:
$$\sum_{\ka \subset \la} c^{\ka}_{\mu,\ze} \cdot c^{\la}_{\ka,\ep} = \sum_{\eta \subset \la} c^{\eta}_{\ze,\ep}\cdot c^{\la}_{\mu,\eta}.$$

\end{proof}

It turns out that for minimal triples $(\lmn)$ of partitions in $\Pn_3$, we can always decompose $\nu$ (or, equivalently, $\mu$) into subpartitions $\nu'$ and $\nu''$ such that all the coefficients that appear in the expression
$$\sum_{\ka \subset \la} c^{\ka}_{\mu,\nu'} \cdot c^{\la}_{\ka,\nu''} = \sum_{\eta \subset \la} c^{\la}_{\mu,\eta}\cdot c^{\eta}_{\nu',\nu''}$$
 are indexed by minimal triples. We can solve this equation for $\clmn$, and we call the resulting expression a \emph{minimal path}.

In particular, we can pick $\nu''$ to consist of a single row or column. In this case, coefficients involving $\nu''$ can be obtained using the Pieri rule. Since $\nu'$ is strictly smaller than $\nu$, we can apply our inductive hypothesis or results of a previous case to compute $c^{\ka}_{\mu,\nu'}$ for $\ka \subset \la$. On the other hand, since $|\eta|=|\nu|$, we use the following lemma to simplify coefficients of the form $c^{\la}_{\mu,\eta}$ for $\eta \subset \la,\eta \neq \nu$.

\begin{lem} $\clmn=c^{\la-\om^\mu_3-\om^\nu_3}_{\mu-\om^\mu_3,\nu-\om^\nu_3}$.
\label{lem:3to2}
\end{lem}

\begin{proof} We first use Lemma \ref{lem:path} with $\ep=(1^3)$ and $\ze = \nu - \ep$, to get 
$\clmn \cdot c^{\nu}_{\ze,\ep} = c^{\la-\ep}_{\mu,\ze} \cdot c^\la_{\la-\ep,\ep}.$ By the Pieri rule, $c^{\nu}_{\ze,\ep}=c^\la_{\la-\ep,\ep}=1$, and so we get $\clmn = c^{\la-\ep}_{\mu,\ze}.$ We can then iterate this to get $\clmn = c^{\la-\om^\nu_3}_{\mu,\nu-\om^\nu_3}$. Finally, we use the symmetry between $\mu$ and $\nu$ to obtain our identity.

\end{proof}

Since, in general, $\eta_3 > \nu_3$ and thus $|\om^\eta_3|>|\om^\nu_3|$, this lemma allows us to reduce $c^{\la}_{\mu,\eta}$ such that it can also be computed by our inductive hypothesis or results of a previous case.

\subsection{Main Lemmas} \label{ssec:lems} We will reduce our minimal path expressions to one of the following 2 identities, depending on whether $\nu''$ is taken to be a row or a column in our decomposition of $\nu$.

We will use the notation $[n]$ to denote $\set{1,\ldots,n}$.

\begin{lem}
Let $n$ be fixed, and let $I=[n]$. Given sets $\si=\set{\si_i},\tau=\set{\tau_i}$ indexed by $i \in I$, we can define 
\begin{align*}
\be_j(\si,\tau) &=\set{\si_i-\si_j} \cup \set{\tau_i+\si_j}, \\
\phi_j(x;\si,\tau) &= \phi\left(x;\be_j(\si,\tau)\right), \\
\Phi(x;\si,\tau) &= \sum_{j \in I} \phi_j(x;\si,\tau)
\end{align*}
Then for all $x,\si,\tau$,
$$\Phi(x;\si,\tau) \equiv \Phi(x;\tau,\si).$$
\label{lem:col}
\end{lem}

\begin{proof}
We use induction on $n$. If $n=1$, then $$\be_1(\si,\tau)=\be_1(\tau,\si)=\set{0,\tau_1+\si_1},$$ and so 
$$\Phi(x;\si,\tau) = \Phi(x;\tau,\si) = \frac{\tau_1+\si_1-x}{\tau_1+\si_1}.$$

For greater $n$, we note that each $\phi_j(x;\si,\tau)$ and $\phi_j(x;\tau,\si)$ is a polynomial of degree $2n-1$ in $x$. We will show that the expression $\Phi(x;\si,\tau) - \Phi(x;\tau,\si)$ vanishes at all points of the form $x_{kl}=(\si_k+\tau_l)$, $k,l \in I$, and therefore must be identically $0$. If we fix some $k$ and $l$ in $I$, we see that
$$\phi_k(x_{kl};\si,\tau) = \phi_l(x_{kl};\tau,\si)=0.$$ 
If $j \neq k$, then 
$$\phi_j(x_{kl};\si,\tau) =\prod_{i \neq j} \frac{\si_i-\si_j-\si_k-\tau_l}{\si_i-\si_j}\prod_{i} \frac{\tau_i+\si_j-\si_k-\tau_l}{\tau_i+\si_j}.$$
We factor out the $i=k$ term from the first product and the $i=l$ term from the second product to get 
$$\phi_j(x_{kl};\si,\tau) = \frac{(-\si_j-\tau_l)}{(\si_k-\si_j)}\frac{(\si_j-\si_k)}{(\tau_l+\si_j)}\left(\prod_{i \neq j,k} \frac{\si_i-\si_j-\si_k-\tau_l}{\si_i-\si_j}\prod_{i \neq l} \frac{\tau_i+\si_j-\si_k-\tau_l}{\tau_i+\si_j} \right).$$
Since $$\frac{(-\si_j-\tau_l)}{(\si_k-\si_j)}\frac{(\si_j-\si_k)}{(\tau_l+\si_j)} = 1,$$
we get that
\begin{align*}
\phi_j(x_{kl};\si,\tau) &=\prod_{i \neq l} \frac{\tau_i+\si_j+\si_k-\tau_l}{\tau_i+\si_j} \prod_{i \neq j,k} \frac{\si_i-\si_j-\si_k-\tau_l}{\si_i-\si_j} \\
&=\phi_j(x_{kl};\si_{i\neq k},\tau_{i\neq l}).
\end{align*}

By a similar calculation, we have 
$$\phi_j(x_{kl};\tau,\si) = \phi_j(x_{kl};\tau_{i\neq l},\si_{i\neq k}).$$
Therefore
$$\Phi(x_{kl};\si,\tau) - \Phi(x_{kl};\tau,\si) = \Phi(x_{kl};\si_{i\neq k},\tau_{i\neq l}) - \Phi(x_{kl};\tau_{i\neq l},\si_{i\neq k}),$$
which is identically 0, by the inductive hypothesis.

\end{proof}

\begin{lem} Fix $n$ and let $\si=(\si_1,\si_2),\tau=(\tau_1,\tau_2)$. Let $\si^k_i$ denote $\si_i+k$. Let
\begin{align*}
\be^n_t(j;\si,\tau) &=\bigcup_{\substack{i \in [2]\\ k \in [n-t]}} \set{-k}\cup\set{\si^{k-1}_i-\si^{t}_j: i \neq j} \cup\set{\tau^{k-1}_i+\si^{t}_j}, \\
\phi^n_t(x;\si,\tau) &= \phi\left(x;\be^n_t(1;\si,\tau)\right) \cdot \phi\left(x;\be^n_{n-t}(2;\si,\tau)\right), \\
\Phi_n(x;\si,\tau) &=\sum^n_{t=0} \phi^n_t(x;\si,\tau)
\end{align*}
Then for all $x,\si,\tau$,
\begin{equation*}
\Phi_n(x;\si,\tau)  \equiv \Phi_n(x;\tau,\si).
\end{equation*}
\label{lem:row}
\end{lem}

\begin{proof}
We prove this identity by induction on $n$. When $n=1$, the result follows from Lemma \ref{lem:col}. 

For general $n$, we note that each $\phi^n_t(x;\si,\tau)$ and $\phi^n_t(x;\tau,\si)$ is a polynomial of degree $4n$ in $x$, so we must show that $\Phi_n(x;\si,\tau) - \Phi_n(x;\tau,\si)$ vanishes at $4n+1$ points. Note that all terms vanish at $x=1$ since $1$ is contained in at least one of the sets $[n-t]$ or $[t]$. We will show that $\Phi_n(x;\si,\tau) - \Phi_n(x;\tau,\si)$ also vanishes at the $4n$ points given by $x=(\si_l+\tau_m+k-1)$, $l,m \in [2], k \in [n]$. Since a transposition of $\si_1$ and $\si_2$ takes $\phi^n_t(x;\si,\tau)$ to $\phi^n_{n-t}(x;\si,\tau)$ and keeps $\phi^n_t(x;\tau,\si)$ fixed, we can assume without loss of generality that $l=m=1$, and so we let $x_k=(\si_1+\tau_1+k-1)$.

We claim that 
\begin{equation}
\phi^n_t(x;\si,\tau)=\left\{\begin{matrix}
0 &\mbox{ if } t<k \\
\phi^{n-k}_{t-k}(x_k;\si+k\mathbf{e}_1,\tau+k\mathbf{e}_1) \cdot c^n_k &\mbox{ if } t\geq k
\end{matrix} \right.
\label{eq:lem2}
\end{equation}
where $\mathbf{e}_1=(1,0)$ and $c^n_k$ is a term that does not depend on $t$, and is symmetric in $\si$ and $\tau$. 

We note that if $t < k$, then for $j=k-t$,
$$\tau^{j-1}_1+\si^t_1-x_k=t+j-k=0.$$
Since $k \in [n]$, $j$ must be in $[n-t]$, and so this implies that for $t < k$
$$\phi^n_t(x;\si,\tau) = 0.$$

Now assume that $t\geq k$. By definition, we have that
$$\phi^n_t(x;\si,\tau) = \phi\left(x;\be^n_t(1;\si,\tau)\right) \cdot \phi\left(x;\be^n_{n-t}(2;\si,\tau)\right),$$
and so we work with each of these two factors separately.

Let $\ab{a}_j=\ab{x_k;a}_j$. We have that
\begin{align*}
\phi\left(x_k;\be^n_t(1;\si,\tau)\right) &= \ab{-n+t}_{n-t} \ab{\si_2-\si_1-t}_{n-t} \ab{\tau_2+\si_1+t}_{n-t} \ab{\tau_1+\si_1+t}_{n-t}, \\
\phi\left(x_k;\be^{n-k}_{t-k}(1;\si+k\mathbf{e}_1,\tau+k\mathbf{e}_1)\right) &= \ab{-n+t}_{n-t} \ab{\si_2-\si_1-t}_{n-t} \ab{\tau_2+\si_1+t}_{n-t} \ab{\tau_1+\si_1+k+t}_{n-t}. 
\end{align*}
Note that the first three factors on the right hand side are the same in both lines. Therefore, if we divide the first expression by the second, we can simplify the ratio using \ref{eq:phid1} to obtain:
\begin{align}
\frac{\phi\left(x_k;\be^n_t(1;\si,\tau)\right)}{\phi\left(x_k;\be^{n-k}_{t-k}(1;\si+k\mathbf{e}_1,\tau+k\mathbf{e}_1)\right)} &= \frac{\ab{\tau_1+\si_1+t}_{n-t}}{\ab{\tau_1+\si_1+k+t}_{n-t}} \nonumber \\
&= \frac{\ab{\tau_1+\si_1+t}_{k}}{\ab{\tau_1+\si_1+n}_{k}}. 
\label{eq:bet1}
\end{align}

On the other hand, we have that:
\begin{align*}
\phi\left(x_k;\be^n_t(2;\si,\tau)\right) =& \ab{-t}_{t}\cdot \ab{\si_1-\si_2-n+t}_{t}\cdot\ab{\tau_1+\si_2+n-t}_{t} \\
&\qquad \cdot\ab{\tau_2+\si_2+n-t}_{t}, \\
\phi\left(x_k;\be^{n-k}_{t-k}(2;\si+k\mathbf{e}_1,\tau+k\mathbf{e}_1)\right) =& \ab{-t+k}_{t-k}\cdot \ab{\si_1-\si_2+k-n+t}_{t-k}\cdot\ab{\tau_1+\si_2+k+n-t}_{t-k} \\
&\qquad \cdot \ab{\tau_2+\si_2+n-t}_{t-k}. 
\end{align*}
Therefore if we divide the first expression by the second, and once again use \ref{eq:phid1} to simplify the ratio, we get
\begin{equation}
\frac{\phi\left(x_k;\be^n_t(2;\si,\tau)\right)}{\phi\left(x_k;\be^{n-k}_{t-k}(2;\si+k\mathbf{e}_1,\tau+k\mathbf{e}_1)\right)} = \ab{-t}_{k} \ab{\si_1-\si_2-n+t}_{k} \ab{\tau_1+\si_2+n-t}_{k}\ab{\tau_2+\si_2+n-k}_{k}.
\label{eq:bet2a}
\end{equation}

By \ref{eq:phid2}, we can rewrite the first term on the right hand side of this expression as
$$\ab{-t}_{k} = \frac{1}{\ab{\tau_1+\si_1+t}_{k}}$$
since $\tau_1+\si_1+t-t=x_k+k-1$. We can also simplify the the middle two terms using the same identity. Since
$$(\si_1-\si_2-n+t) + (\tau_1+\si_2+n-t)=x_k - k + 1,$$
by \ref{eq:phid2}, we get that 
$$\ab{\si_1-\si_2-n+t}_{k} \ab{\tau_1+\si_2+n-t}_{k} = 1.$$
Therefore, we can reduce \ref{eq:bet2a} to 
\begin{equation}
\frac{\phi\left(x_k;\be^n_t(2;\si,\tau)\right)}{\phi\left(x_k;\be^{n-k}_{t-k}(2;\si+k\mathbf{e}_1,\tau+k\mathbf{e}_1)\right)} = \frac{\ab{\tau_2+\si_2+n-k}_{k}}{\ab{\tau_1+\si_1+t}_{k}}.
\label{eq:bet2}
\end{equation}

Finally, we multiply the expressions in \ref{eq:bet1} and \ref{eq:bet2} to get that
\begin{align*}
\frac{\phi^n_{t} (x_k;\si,\tau)}{\phi^{n-k}_{t-k} (x_k;\si+k\mathbf{e}_1,\tau+k\mathbf{e}_1)} &= \frac{\ab{\tau_1+\si_1+t}_{k}}{\ab{\tau_1+\si_1+n}_{k}} \cdot \frac{\ab{\tau_2+\si_2+n-k}_{k}}{\ab{\tau_1+\si_1+t}_{k}} \\
&=\frac{\ab{\tau_2+\si_2+n-k}_{k}}{\ab{\tau_1+\si_1+n}_{k}},
\end{align*}
which completes our proof of equation \ref{eq:lem2}, with 
$$c_k = \frac{\ab{\tau_2+\si_2+n-k}_{k}}{\ab{\tau_1+\si_1+n}_{k}}.$$
It is easy to see that this $c_k$ does not depend on $t$ and is symmetric in $\si$ and $\tau.$

Since each $\phi^{n}_t$ contains this factor of $c_k$ whenever $t\geq k$, it follows that
$$\Phi_n(x_k;\si,\tau) = \Phi_{n-k}(x_k;\si+k\mathbf{e}_1,\tau+k\mathbf{e}_1) c_k.$$
Similarly, by transposing $\si$ and $\tau$, we get
$$\Phi_n(x_k;\tau,\si) = \Phi_{n-k}(x_k;\tau+k\mathbf{e}_1,\si+k\mathbf{e}_1)  c_k.$$
Thus, by the inductive hypothesis:
$$\Phi_n(x_k;\si,\tau)-\Phi_n(x_k;\tau,\si) = 0.$$

\end{proof}

\subsection{Verification of Division Number Formulas} \label{sec:ver} As in Section \ref{sec:div}, we will use the following notation: 
\begin{align*}
\fd_{ijk} &= |\la_i + \mu_j - \nu_k|, \\
\xi_{ij} &= \xi_i - \xi_j, \\
\mathfrak{p} &= \max(\la_3-\mu_2-\nu_3,0), \\
x^{\pm} &= x\pm \mathfrak{p}.
\end{align*}
Given a set of division numbers $\mathfrak{n}$, recall that
\begin{equation}
\mathbf{d}^\la_{\mu,\nu}(\mathfrak{n}) = \prod_{i\leq j} [h_\la^{ij};n^\la_{ij}]\cdot [-h_\mu^{ij};n^\mu_{ij}]\cdot [-h_\nu^{ij};n^\nu_{ij}],
\label{eq:dlmn}
\end{equation}
where $n^\la_{ij}$ is the division number corresponding to the $i^{\mbox{\tiny th}}$ strip in $\om^{\la}_j$, and $n^\xi_{ij}$ is the division number corresponding to the $i^{\mbox{\tiny th}}$ strip in $\om^\xi_{j-1}$ for $\xi \in \set{\mu,\nu}$. 

For each minimal triple $(\lmn)$ we have a set of division numbers $\mathfrak{n}(\lmn)$, as listed in Table \ref{tab:div}. Let 
$$\dlmn = \mathbf{d}^\la_{\mu,\nu}(\mathfrak{n}(\lmn)).$$ We verify that for each case in Table \ref{tab:div}, $\clmn = \dlmn$.

In order to do this, for each fixed triple $(\lmn)$, we first decompose $\nu$ into two subpartitions $\ze$ and $\ep$, such that $\ep$ consists of a single row or column. By Lemma \ref{lem:path}, we then get a minimal path of the form:
\begin{equation}
\sum_{i} c^{\ka(i)}_{\mu,\ze} \cdot c^{\la}_{\ka(i),\ep} = \sum_{i} c^{\eta(i)}_{\ze,\ep} \cdot c^\la_{\mu,\eta(i)},
\end{equation}
where for some $i$, $\eta(i)=\nu$. For such a path, we can use either a previously established result or induction to get that each $c^{\xi_1}_{\xi_2,\xi_3}$ equals $d^{\xi_1}_{\xi_2,\xi_3}$ for all the triples $(\xi_1,\xi_2,\xi_3)\neq(\lmn)$. Therefore, to show that $\clmn=\dlmn$, it suffices to instead verify the analogous identity for $\dlmn$:
\begin{equation}
\sum_{i} d^{\ka(i)}_{\mu,\ze} \cdot d^{\la}_{\ka(i),\ep} = \sum_{i} d^{\eta(i)}_{\ze,\ep} \cdot d^\la_{\mu,\eta(i)}.
\label{eq:did}
\end{equation}

When the sums on either side of equation \ref{eq:did} consist of more than one term, we can prove this identity by writing each $d^{\ka(i)}_{\mu,\ze} \cdot d^{\la}_{\ka(i),\ep}$ and $d^{\eta(i)}_{\ze,\ep} \cdot d^\la_{\mu,\eta(i)}$ as the product of $\dlmn \cdot d^\nu_{\ze,\ep}$ with additional terms produced by changes to the anchors and division numbers, as given by equation \ref{eq:phid1}. We will then show that the additional factors produced satisfy equation \ref{eq:did} by showing that they fall into the form of Lemma \ref{lem:col} (if $\ep$ is a single column) or \ref{lem:row} (if $\ep$ is a single row). 

In the proofs below, we will also use the fact that since $\clmn = c^\la_{\nu,\mu}$, any division number formula that we prove for a coefficient based on a condition involving $\mu$ and $\nu$ can be subsequently be used for the condition obtained by interchanging $\mu$ and $\nu$, as long as one also interchanges the role of $\mu$ and $\nu$ in the formula.

\begin{prop}
If $\mu_3=\mu_2$, then $c^{\la}_{\mu,\nu}$ is given by 
$$\la: \begin{bmatrix}
\fd_{333} & \fd_{222} & 0 \\
\fd_{333} & \fd_{222} & \\
\fd_{333} &&
\end{bmatrix} \qquad 
\mu: \begin{bmatrix}
0 & \mathfrak{d}_{111}\\
0 & 
\end{bmatrix} \qquad 
\nu: \begin{bmatrix}
\fd_{333} & \fd_{222} \\
\fd_{333} & 
\end{bmatrix}$$
\label{prop:horn1}
\end{prop}

\begin{proof} By Lemma \ref{lem:3to2}, $\clmn=c^{\la-\om^{\mu}_3}_{\mu-\om^{\mu}_3,\nu}.$ Since, $\mu-\om^{\mu}_3$ consists of a single part, we can use the Pieri rule to compute $c^{\la-\om^{\mu}_3}_{\mu-\om^{\mu}_3,\nu}.$
The Pieri rule for rows can be obtained from Theorem \ref{thm:jpieri} and equation \ref{eq:transpose}. In particular, we get that $c^{\la-\om^{\mu}_3}_{\mu-\om^{\mu}_3,\nu}$ is obtained by treating all the hooks in $\lmn$ as flipped hooks, except those corresponding to boxes $(i,j) \in \nu,\la$ such that $\nu_i'=\la_i'$ and $\nu_j < \la_j$. Finally, we note that the $\fd_{131}$ flipped hooks in $\om^{\la}_1$ can be exchanged with the last $\fd_{131}$ flipped hooks in $\om^{\mu}_1$, leaving no flipped hooks in $\om^{\la}_1$ and only $\mu_{13}-\fd_{131}=\fd_{111}$ flipped hooks in $\om^\mu_1$. 
\end{proof}

\begin{prop}
If $\mu_2=\mu_1$, then $c^{\la}_{\mu,\nu}$ is given by:
$$\la: \begin{bmatrix}
\fd_{111} & \fd_{213} & 0 \\
\fd_{222} & \fd_{213} & \\
\fd_{111} &&
\end{bmatrix} \qquad
\mu:\begin{bmatrix}
0 & 0\\
\fd_{333} & 
\end{bmatrix} \qquad
\nu: \begin{bmatrix}
\fd_{213} & \fd_{111} \\
\fd_{213} & 
\end{bmatrix}$$ 
\label{prop:horn2}
\end{prop}

\begin{proof} Let $\dlmn$ by the hypothesized formula. We use induction on $(\nu_2-\nu_3)$ to show that $\clmn=\dlmn$. If $\nu_2=\nu_3$, we can determine $\clmn$ using Prop. \ref{prop:horn1} with the roles of $\mu$ and $\nu$ reversed in the following way. First, note that by Horn inequality (8),  $\la_2 \leq \mu_1+\nu_2=\mu_2+\nu_2$, and by inequality (15), $\la_2 \geq \mu_2+\nu_3 = \mu_2+\nu_2$, so $\la_2=\mu_2+\nu_2$. Therefore, $\fd_{222}=0$, and so $\fd_{213}=\fd_{222}=0$. Since $|\la|=|\mu|+|\nu|$, it also follows that $\fd_{111}=\fd_{333}$. Finally, note that since $\la_2-\mu_3-\nu_3=\mu_1-\mu_3$ in this case, we can exchange the lower hooks in the second strip of $\om^\la_3$ with the upper hooks in the first strip of $\om^\mu_2$, so that Prop. \ref{prop:horn1} gives us the following division numbers:
$$\la: \begin{bmatrix}
\fd_{111} & 0 & 0 \\
0 & 0 & \\
\fd_{111} &&
\end{bmatrix} \qquad
\mu:\begin{bmatrix}
0 & 0\\
\fd_{333} & 
\end{bmatrix} \qquad
\nu: \begin{bmatrix}
0 & \fd_{111} \\
0 & 
\end{bmatrix}$$ 
We can verify that this is the same as $\dlmn$ in this case.

If $\nu_2 > \nu_3$ then we can decompose $\nu$ into $\ep=(1,1)$ and $\ze=\nu-\ep$. Then by Lemma \ref{lem:path}, we have 
\begin{equation}
\sum_{1\leq i \leq 3} c^{\la}_{\mu,\eta(i)} c^{\eta(i)}_{\ze,\ep} = \sum_{1\leq i \leq 3} c^{\ka(i)}_{\mu,\ze} c^{\la}_{\ka(i),\ep},
\label{eq:mu12path}
\end{equation}
where $\eta(i) = \ze + (1,1,1) - \mathbf{e}_i$ and $\ka(i)=\la-(1,1,1)+\mathbf{e}_i$, where $\mathbf{e}_i$ is a triple consisting of a $1$ in the $i^{\mbox{\tiny th}}$ position and $0$'s elsewhere. Note that $$\eta(3)=\nu.$$

We will show that our hypothesized coefficients satisfy equation \ref{eq:mu12path}. We can determine $d^{\la}_{\ka(i),\ep}$ and $d^{\eta(i)}_{\ze,\ep}$ by the Pieri rule, since $\ep$ consists of a single column. In particular, we have that
\begin{align*}
d^{\la}_{\ka(i),\ep} &=\phi\left(1-r;\set{h^{ij}_{\la}+1,-h^{ij}_{\ka(i)}+1:j>i}\right),\\
d^{\eta(i)}_{\ze,\ep} &=\phi\left(1-r;\set{h^{ij}_{\eta(i)}+1,-h^{ij}_{\ze}+1:j>i}\right).
\end{align*}
Note that by this definition, $d^{\nu}_{\ze,\ep}=d^{\eta(3)}_{\ze,\ep}=1$.

We can also write out $d^{\ka(i)}_{\mu,\ze}$ and $d^{\la}_{\mu,\eta(i)}$, by comparing them to $\dlmn$, since $\ka(i)$ is obtained by modifying the parts of $\la$ and $\eta(i)$ and $\ze$ are obtained by modifying the parts of $\nu$. Therefore, $d^{\ka(i)}_{\mu,\ze}$ and $d^{\la}_{\mu,\eta(i)}$ can be determined by examining how these changes to $\la$ and $\nu$ change the anchors and division numbers for each strip.

In the symbols below, the entries denote how the corresponding anchor for each strip must be changed for that coefficient compared to the anchor of $\dlmn$, and a $^*$ indicates a change of $-1$ to the corresponding division number. 

\begin{alignat*}{2}
\ka(1): \begin{pmatrix}
+1^* & +1^* & 0 \\
0 & 0^* & \\
0^* &
\end{pmatrix} \qquad
&\mu:\begin{pmatrix}
0 & 0\\
0^* & 
\end{pmatrix} \qquad
&\ze: \begin{pmatrix}
+1^* & 0^* \\
+1^* & 
\end{pmatrix}, \\
\ka(2): \begin{pmatrix}
0 & -1 & 0 \\
+1^* & 0 & \\
0 &
\end{pmatrix} \qquad
&\mu:\begin{pmatrix}
0 & 0\\
0^* & 
\end{pmatrix} \qquad
&\ze: \begin{pmatrix}
+1 & 0 \\
+1 &  
\end{pmatrix}, \\
\ka(3): \begin{pmatrix}
-1 & 0^* & 0 \\
-1 & 0^* & \\
0 &
\end{pmatrix} \qquad
&\mu:\begin{pmatrix}
0 & 0\\
0 & 
\end{pmatrix} \qquad
&\ze: \begin{pmatrix}
+1^* & 0 \\
+1^* & 
\end{pmatrix}, \\
\la: \begin{pmatrix}
0^* & 0^* & 0 \\
0 & 0^* & \\
0^* &&
\end{pmatrix} \qquad
&\mu:\begin{pmatrix}
0 & 0\\
0^* & 
\end{pmatrix} \qquad
&\eta(1): \begin{pmatrix}
+2^* & +1^* \\
+1^* & 
\end{pmatrix}, \\
\la: \begin{pmatrix}
0 & 0^* & 0 \\
0^* & 0^* & \\
0 &&
\end{pmatrix} \qquad
&\mu:\begin{pmatrix}
0 & 0\\
0^* & 
\end{pmatrix} \qquad
&\eta(2): \begin{pmatrix}
+1^* & -1 \\
+2^* & 
\end{pmatrix}.
\end{alignat*}

This allows us to determine each summand $d^{\ka(i)}_{\mu,\ze} d^{\la}_{\ka(i),\ep}$ and $d^{\la}_{\mu,\eta(i)} d^{\eta(i)}_{\ze,\ep}$ of equation \ref{eq:mu12path} compared to $\dlmn$, since we can use equations \ref{eq:mod1} and \ref{eq:mod2} to write terms of the form $[b+r;n+s]$ as a product of $[b;n]$ and some additional factors.

In particular, we factor out $\dlmn$ from each of these terms. In addition, we factor out terms that appear in a majority of the six summands. Note that these terms come from the blocks $\om^\xi_2$ for each partition $\xi$. Thus, we factor out $\frac{1}{\mathcal{X}}$ from each expression, where
$$\mathcal{X}=\phi\left(1-r; {\la_{12}+r+\fd_{213},\fd_{213},-\mu_{23}-r+\fd_{333},1-\nu_{13}-2r,1-\nu_{23}-r}\right).$$
Using the notation of Lemma \ref{lem:col}, we can rewrite $\mathcal{X}$ as $\phi_3(1-r;\si,\tau)$, where
\begin{alignat*}{3}
\tau_1 &= \la_{12}+\fd_{213}+r &&= \la_1 - \mu_1 - \nu_3 + r \\
\tau_2 &= \fd_{213} &&= \la_2 - \mu_1 - \nu_3  \\
\tau_3 &= \fd_{333}-\mu_{23}-r &&= \la_3 - \mu_1 - \nu_3 - r \\
\si_1 &= 1-\nu_{13}-2r \\
\si_2 &= 1-\nu_{23}-r \\
\si_3 &= 0.
\end{alignat*}

This allows us to write each term $d^{\ka(i)}_{\mu,\ze} d^{\la}_{\ka(i),\ep}$ and $d^{\la}_{\mu,\eta(i)} d^{\eta(i)}_{\ze,\ep}$ as a product of the form $$\frac{\dlmn}{\mathcal{X}} \phi(1-r;\set{a_1,\ldots,a_n}).$$ In the table below, we present the elements of the set $\mathcal{A}$ corresponding to each term. We will use $\ka(i)$ to indicate terms corresponding to $d^{\ka(i)}_{\mu,\ze} d^{\la}_{\ka(i),\ep}$ and $\eta(i)$ to indicate terms corresponding $d^{\la}_{\mu,\eta(i)} d^{\eta(i)}_{\ze,\ep}$.

\begin{center}
\begin{tabular}{|c|c|c|c|c|c|}
\hline
& $a_1$ & $a_2$ & $a_3$ & $a_4$ & $a_5$\\
\hline
$\ka(1)$  &  $\la_{12}+\fd_{213}+r$&$1-\fd_{111}-r$&$-\la_{13}-2r$&$-\la_{12}-r$&$1+\nu_{12}-\fd_{111}$ \\
$\ka(2)$  &  $\fd_{213}$&$-\la_{23}-r$&$\la_{12}+r$&$\fd_{213}+1-\nu_{13}-2r$&$1-\nu_{23}-r$ \\
$\ka(3)$  & $\la_{13}+2r$&$\la_{23}+r$&$\fd_{333}-\mu_{23}-r$&$1-\la_{13}-\fd_{111}-3r$&$1-\la_{23}-\fd_{222}-2r$ \\
$\eta(1)$  & $1-\la_{13}-\fd_{111}-3r$&$1-\fd_{111}-r$&$\nu_{13}+2r-1$&$\nu_{12}+r$&$\fd_{213}+1-\nu_{13}-2r$\\
$\eta(2)$  &  $1-\la_{23}-\fd_{222}-2r$&$1-\fd_{111}+\nu_{12}$&$\nu_{23}+r-1$&$-\nu_{12}-r$&$\fd_{213}+1-\nu_{23}-r$ \\
\hline
\end{tabular}
\end{center}

Using the $\si_i$ and $\tau_i$ defined above, and once again using the notation of Lemma \ref{lem:col}, one can check that
\begin{align*}
d^{\ka(i)}_{\mu,\ze} d^{\la}_{\ka(i),\ep} &= \frac{\dlmn}{\mathcal{X}} {\phi_i(1-r;\tau,\si)}, \\
d^{\la}_{\mu,\eta(i)} d^{\eta(i)}_{\ze,\ep} &= \frac{\dlmn}{\mathcal{X}} {\phi_i(1-r;\si,\tau)}.
\end{align*}
Therefore, we have that
\begin{align*}
\sum^{3}_{i=1} d^{\ka(i)}_{\mu,\ze} c^{\la}_{\ka(i),\ep} &= \frac{\dlmn}{\mathcal{X}} {\Phi(1-r;\tau,\si)}, \\
\sum^{3}_{i=1} d^{\la}_{\mu,\eta(i)} c^{\eta(i)}_{\ze,\ep} &= \frac{\dlmn}{\mathcal{X}} {\Phi(1-r;\si,\tau)}.
\end{align*}
and so by Lemma \ref{lem:col}, we have that both sums are equal, showing that $\dlmn$ satisfies equation \ref{eq:mu12path}. 

By the inductive hypothesis, each $c^{\ka(i)}_{\mu,\ze}=d^{\ka(i)}_{\mu,\ze}$, since $\mu$ is unchanged and $\ze_2-\ze_3=\nu_2-\nu_3-1$. Similarly, $c^{\la}_{\mu,\eta(i)}=d^{\la}_{\mu,\eta(i)}$ for $i=1,2$, since $\eta(1)_2 - \eta(1)_3 = \nu_2 - \nu_3 - 1$ and $\eta(2)_2 - \eta(2)_3 = \nu_2 - \nu_3 - 2$. Therefore, we get that the remaining term $c^{\la}_{\mu,\eta(i)} = \clmn$ must equal $\dlmn$.

\end{proof}

These last two propositions are instrumental in proving all of the remaining cases, since they allow us to form minimal paths by decomposing $\nu$ into two pieces $\ze$ and $\ep$ such that $\ep$ consists of a single part, and $\ze$ is such that either $\ze_2=\ze_3$ or $\ze_1=\ze_2$. Then we can determine coefficients involving $\ep$ using the Pieri rule for rows, and coefficients involving $\ze$ using either Proposition \ref{prop:horn1} or \ref{prop:horn2}.

In particular, Prop \ref{prop:horn1} allows us to form a minimal path by decomposing $\nu$ into $\ep=(\nu_2-\nu_3)$ and $\ze=(\nu_1,\nu_3,\nu_3)$.
$$\young(~~~~~~~,~~\bullet\bullet,~~) = \yng(7,2,2) + \young(\bullet\bullet)$$
Prop \ref{prop:horn2} allows us to form a minimal path by decomposing $\nu$ into $\ep=\om^\nu_1$ and $\ze=\nu-\ep=\om^\nu_3+\om^\nu_2$.
$$\young(~~~~\bullet\bullet\bullet,~~~~,~~) = \yng(4,4,2) + \young(\bullet\bullet\bullet)$$

\begin{prop} If $\la_3=\mu_2+\nu_2,$ then $\clmn$ is given by the following division numbers:
$$\la:\begin{bmatrix}
\fd_{333} & \fd_{222} & 0 \\
\fd_{333} & \fd_{222} & \\
0 &&
\end{bmatrix} \qquad
\mu:\begin{bmatrix}
\fd_{332} & \fd_{223}\\
0 & 
\end{bmatrix} \qquad
\nu:\begin{bmatrix}
\fd_{323} & \fd_{232} \\
0 & 
\end{bmatrix}.$$
\label{prop:la3}
\end{prop}

\begin{proof} We will first use Lemma \ref{lem:path} to show that $$\clmn c^{\nu}_{\ze,\ep}=c^{\ka}_{\mu,\ze} c^{\la}_{\ka,\ep},$$ where $\ep=(\nu_2-\nu_3),\ze=(\nu_1,\nu_3,\nu_3)$ and $\ka=\la-(0,0,\nu_2-\nu_3).$ For any $\eta$ such that $c^{\eta}_{\ep,\ze} \neq 0$, we have that $\eta_2 \leq \nu_2$, with equality holding only if $\eta=\nu$. By Horn inequality (11), we know that if $\eta_2+\mu_2<\nu_2+\mu_2 =\la_3$, then $c^\la_{\mu,\eta}=0$. On the other hand, if $\ka_3 > \mu_2+\nu_3$ then since $\ze_2=\ze_3=\nu_3$, we have that $\ka_3 >\mu_2+\ze_2$, which would imply that $c^\ka_{\mu,\ze}=0$.  

Using Prop \ref{prop:horn1} and the fact that $\ka=(\la_1,\la_2,\la_3-\nu_{23})$ and $\ze=(\nu_1,\nu_3,\nu_3)$, we have that $c^{\ka}_{\mu,\ze}$ is given by the division numbers:
$$\ka: \begin{bmatrix}
\fd_{333}-\nu_{23} & \fd_{222}+\nu_{23} & 0 \\
\fd_{333}-\nu_{23} & \fd_{222}+\nu_{23} & \\
\fd_{333}-\nu_{23} &&
\end{bmatrix} \qquad 
\mu: \begin{bmatrix}
\fd_{333}-\nu_{23} & \fd_{222}+\nu_{23} \\
\fd_{333}-\nu_{23} & 
\end{bmatrix} \qquad 
\ze: \begin{bmatrix}
0 & \mathfrak{d}_{111}\\
0 & 
\end{bmatrix}$$
where $\fd_{ijk}$ still refers to $|\la_i-\mu_j-\nu_k|$. Therefore,
\begin{align*}
c^{\ka}_{\mu,\ze} = [\la_{13}+&\nu_{23}+2r;\fd_{333}-\nu_{23}][\la_{23}+\nu_{23}+r;\fd_{333}-\nu_{23}][0;\fd_{333}-\nu_{23}] \\
&\cdot [\la_{12}+r;\fd_{222}-\nu_{23}][0;\fd_{222}-\nu_{23}][-\mu_{13}-2r;\fd_{333}-\nu_{23}] \\
&\cdot [-\mu_{23}-r;\fd_{333}-\nu_{23}][-\mu_{12}-r;\fd_{222}-\nu_{23}][-\nu_{13}-r;\fd_{111}].
\end{align*}
Note that since $\la_3=\mu_2+\nu_2$ in this case, we also have that $\fd_{333}-\nu_{23} = \mu_{23}$. This implies that $$[-\mu_{23}-r;\fd_{333}-\nu_{23}][0;\fd_{333}-\nu_{23}]=1,$$ and so these terms can be removed from the product above.

Next, we compute $c^{\la}_{\ka,\ep}$ and $c^{\nu}_{\ze,\ep}$ by the Pieri rule to get that
\begin{align*}
c^{\la}_{\ka,\ep} &= [\la_{13}+2r;\nu_{23}][\la_{23}+r;\nu_{23}][-\la_{13}+\nu_{23}+r;\nu_{23}][-\la_{23}+\nu_{23};\nu_{23}], \\
c^{\nu}_{\ze,\ep} &= [\nu_{12}+r;\nu_{23}][-\nu_{13}-r;\nu_{23}].
\end{align*}

Therefore, we get that 
\begin{align*}
c^{\ka}_{\mu,\ze}c^{\la}_{\ka,\ep} &= [\la_{13}+\nu_{23}+2r;\fd_{333}-\nu_{23}][\la_{23}+\nu_{23}+r;\fd_{333}-\nu_{23}][\la_{12}+r;\fd_{222}-\nu_{23}] \\
&\qquad\qquad \cdot [0;\fd_{222}-\nu_{23}][-\mu_{13}-2r;\fd_{333}-\nu_{23}][-\mu_{12}-r;\fd_{222}-\nu_{23}] \\
&\qquad\qquad \cdot [-\nu_{13}-r;\fd_{111}][\la_{13}+2r;\nu_{23}][\la_{23}+r;\nu_{23}] \\
&\qquad\qquad \cdot [-\la_{13}+\nu_{23}+r;\nu_{23}][-\la_{23}+\nu_{23};\nu_{23}] \\
&= [\la_{13}+2r;\fd_{333}][\la_{23}+r;\fd_{333}][\la_{12}+r;\fd_{222}][0;\fd_{222}][-\mu_{13}-2r;\fd_{332}] \\
&\qquad\qquad \cdot[-\mu_{12}-r;\fd_{223}][-\nu_{13}-r;\fd_{111}],
\end{align*}
and so dividing by $c^{\nu}_{\ze,\ep}$ gives us that
\begin{align*}
\clmn &= [\la_{13}+2r;\fd_{333}][\la_{23}+r;\fd_{333}][\la_{12}+r;\fd_{222}][0;\fd_{222}][-\mu_{13}-2r;\fd_{332}] \\
& \qquad \quad \cdot [-\mu_{12}-r;\fd_{223}][-\nu_{13}-r;\fd_{232}][-\nu_{12};\nu_{23}].
\end{align*}
This expression corresponds to the desired division numbers.
\end{proof}

\begin{prop} If $\la_i=\mu_i+\nu_1,$ then $\clmn=c^{\ka}_{\mu,\ze} c^{\la}_{\ka,\ep}$, where $\ep=\om^\nu_1,\ze=\nu-\ep$ and $\ka=\la - |\ep| \mathbf{e}_i.$
\label{prop:nu1}
\end{prop}

\begin{proof} We apply Lemma \ref{lem:path}. For any $\eta$ such that $c^{\eta}_{\ep,\ze} \neq 0$, we have that $\eta_1 \leq \nu_1$, with equality holding only if $\eta=\nu$. By Horn inequalities (7),(9) and (12), we know that if $\eta_1+\mu_i<\nu_1+\mu_i =\la_i$, then $c^\la_{\mu,\eta}=0$. On the other hand, if $\ka_i > \la_i - |\ep|$, then $\ze_1+\mu_i = \nu_1 - |\ep| +\mu_i =\la_i - |\ep| <\ka_i$, which would imply that $c^\ka_{\mu,\ze}=0$.  

\end{proof}

\begin{prop} If $\la_i=\mu_i+\nu_3,$ then $\clmn=c^{\ka}_{\mu,\ze} c^{\la}_{\ka,\ep}$, where $\ep=\om^\nu_1,\ze=\nu-\ep$ and $\ka=\mu+(\nu^3_3)+(\nu^3_2)-\nu_2 \mathbf{e}_i.$
\label{prop:nu3}
\end{prop}

\begin{proof} We apply Lemma \ref{lem:path}. For any $\eta$ such that $c^{\eta}_{\ep,\ze} \neq 0$, we have that $\eta_3 \geq \nu_3$, with equality holding only if $\eta=\nu$. By Horn inequalities (13), (15) and (18), we know that if $\eta_3+\mu_i>\nu_3+\mu_i =\la_i$, then $c^\la_{\mu,\eta}=0$. On the other hand, if $\ka_i < \mu_i+\nu_3,$ then $\ka_i<\mu_i+\ze_3$, which would imply that $c^\ka_{\mu,\ze}=0$.  

\end{proof}

Note that Propositions \ref{prop:nu1} and \ref{prop:nu3} completely determine the coefficients in those cases, since $c^{\la}_{\ka,\ep}$ can be determined by Propositon \ref{prop:horn1} and $c^{\ka}_{\mu,\ze}$ can be determined by Propositon \ref{prop:horn2}. The details are similar to the proof of Proposition \ref{prop:la3}.

The remaining cases all make use of Lemma \ref{lem:row}. 

\begin{prop} If $\la_3=\la_2,$ then $\clmn$ is given by
$$\la: \begin{bmatrix}
\fd_{111} & 0 & 0 \\
\fd_{333} & 0 & \\
\fd_{223} &&
\end{bmatrix} \qquad
\mu:\begin{bmatrix}
\fd_{232} & \fd_{223}\\
\fd_{222} & 
\end{bmatrix} \qquad
\nu:\begin{bmatrix}
\fd_{223} & \fd_{232} \\
\fd_{223} & 
\end{bmatrix}$$
\label{prop:la23}
\end{prop}

\begin{proof} Let $\dlmn$ be the expression given by the proposition. We use induction on $k=\fd_{222}=\nu_2+\mu_2-\la_2$ to show that $\clmn=\dlmn$. If $k=0$, then $\nu_2+\mu_2=\la_2=\la_3$, and so $\clmn$ can be obtained from Prop. \ref{prop:la3} as follows. We observe that the fact that $\la_2=\la_3$ implies that $\fd_{2jk}=\fd_{3jk}$ for all $j,k$, and the fact that $\fd_{222}=0$ implies that $\fd_{111}=\fd_{333}$ and $\fd_{223}=\nu_{23}$. Thus, Prop \ref{prop:la3} gives us
$$\la:\begin{bmatrix}
\fd_{111} & 0 & 0 \\
\fd_{333} & 0 & \\
0 &&
\end{bmatrix} \qquad
\mu:\begin{bmatrix}
\fd_{232} & \fd_{223}\\
0 & 
\end{bmatrix} \qquad
\nu:\begin{bmatrix}
\fd_{223} & \fd_{232} \\
0 & 
\end{bmatrix}.$$
We can verify that this is the same as the desired formula, by noting that since $\fd_{223}=\nu_{23}$, we can swap all the lower hooks in the third strip of $\om^{\la}_3$ with all the upper hooks in the second strip of $\om^{\nu}_2$.

For $k>0$, let 
\begin{align*}
\ep&=(\nu_2-\nu_3),\\
\ze&=(\nu_1,\nu_3,\nu_3),
\end{align*}
and 
\begin{align*}
\eta(t)&= \nu-(-t,t,), \\
\ka(t)&=(\mu_1+k-t,\la_2-\nu_3,\mu_3+t).
\end{align*}
Note that $\eta(0)=\nu.$ 

By Lemma \ref{lem:path}, we have that 
\begin{equation}
\sum^k_{t=0} c^{\ka(t)}_{\mu,\ep} \cdot c^{\la}_{\ka(t),\ze} = \sum^k_{t=0} c^{\eta(t)}_{\ze,\ep}\cdot c^{\la}_{\mu,\eta(t)}.
\label{eq:la23path}
\end{equation}
We show that our hypothesized coefficients satisy equation \ref{eq:la23path}. 

We can determine $d^{\ka(t)}_{\mu,\ep}$ and $d^{\eta(t)}_{\ze,\ep}$ using the Pieri rule, to get
\begin{align*}
d^{\ka(t)}_{\mu,\ep} &= [\mu_{12}+k-2t+2r;t][\fd_{333}-t+r;t][0;t][\fd_{213}+k-t+r;\nu_{23}-k][0;\nu_{23}-k]\\
&\qquad\qquad \cdot[-\mu_{13}-2r;t][-\mu_{23}-r;t][-\mu_{12}-r;\nu_{23}-k][-\nu_{23}+r;\nu_{23}-k+t], \\
d^{\eta(t)}_{\ze,\ep} &= [\nu_{12}+2t+r;\nu_{23}-t][0;\nu_{23}-t][-\nu_{13}-r;\nu_{23}-t][-\nu_{23}-r;\nu_{23}-t].
\end{align*}

We can obtain $d^{\la}_{\ka(t),\ze}$ using Proposition \ref{prop:horn1}, which gives us
\begin{align*}
d^{\la}_{\ka(t),\ze} &= [\la_{13}+2r;\fd_{333}-t][r;\fd_{333}-t][-\mu_{13}-k+2t-2r;\fd_{333}-t][-\nu_{13}-r;\fd_{333}-t].
\end{align*}

Finally, we can determine each $d^{\la}_{\mu,\eta(t)}$ by modifying the $\dlmn$ to get that:
\begin{align*}
d^{\la}_{\mu,\eta(t)} =& [\la_{13}+2r;\fd_{111}+t][r;\fd_{333}][0;\fd_{223}][-\mu_{13}-2r;\fd_{232}+t][-\mu_{23}-r;\fd_{222}-t] \\
&\qquad \cdot [-\mu_{12}-r;\fd_{223}][-\nu_{13}-t-2r;\fd_{223}][-\nu_{23}+t-r;\fd_{223}] \\
&\qquad \cdot [-\nu_{12}-2t-r;\fd_{232}+t].
\end{align*}

Therefore, using equations \ref{eq:mod1} and \ref{eq:mod2}, we get that:
\begin{align*}
\frac{d^{\la}_{\mu,\eta(t)} d^{\eta(t)}_{\ze,\ep}}{\dlmn d^{\nu}_{\ze,\ep}}
&= \frac{[\fd_{311}+2r;t][-\fd_{212}-2r;t][\fd_{221}+r;t][-\nu_{12}-2t-r;t][-k-r;t]}{[-\mu_{23}-r+k-t;t][-\fd_{331}-t-r;t][\nu_{12}+r;t][\nu_{12}+t+r;t][-t-r;t]}, \\
\frac{d^{\ka(t)}_{\mu,\ep} d^{\la}_{\ka(t),\ze}} {\dlmn d^{\nu}_{\ze,\ep}}
&= \frac{[-\mu_{13}-2r;t][\mu_{12}+k-2t+2r;t][-\mu_{23}-r;t][-k-r;t]}
{[\fd_{212}+r-t;t][-\mu_{13}-k+t-2r;t][-\fd_{331}-t-r;t][-t-r;t]} \\
&\qquad \qquad \cdot \frac{[\fd_{311}+2r;k-t][\mu_{23}-k;k]}{[-\mu_{13}-2r;k-t][\nu_{12}+r-k;k]}.
\end{align*}

Thus, using the notation of Lemma \ref{lem:row}, we have that:
\begin{align*}
d^{\la}_{\mu,\eta(t)}\cdot d^{\eta(t)}_{\ze,\ep}
&= d^{\la}_{\mu,\nu} \cdot d^{\nu}_{\ze,\ep} \cdot \frac{\phi^k_t(x;\si,\tau)}{\phi^k_0(x;\si,\tau)}, \\
d^{\ka(t)}_{\mu,\ep} \cdot d^{\la}_{\ka(t),\ze} 
&= d^{\la}_{\mu,\nu} \cdot d^{\nu}_{\ze,\ep} \cdot \frac{\phi^k_t(x;\tau,\si)}{\phi^k_0(x;\si,\tau)},
\end{align*}
where $\si=(0,\fd_{321}+r)$ and $\tau=(-\fd_{331}-r,-\fd_{311}-3r)$. 


This implies that 
\begin{align}
\sum^k_{t=0} d^{\eta(t)}_{\ze,\ep}\cdot d^{\la}_{\mu,\eta(t)} &= \frac{d^{\la}_{\mu,\nu} \cdot d^{\nu}_{\ze,\ep}}{\phi^k_0(x;\si,\tau)} \cdot {\Phi_k(x;\si,\tau)}, \\
\sum^k_{t=0} d^{\ka(t)}_{\mu,\ep} \cdot d^{\la}_{\ka(t),\ze} &= \frac{d^{\la}_{\mu,\nu} \cdot d^{\nu}_{\ze,\ep}}{\phi^k_0(x;\si,\tau)} \cdot {\Phi_k(x;\tau,\si)}.
\end{align}
Therefore, by Lemma \ref{lem:row}, we have that the hypothesized coefficients satisfy equation \ref{eq:la23path}.

Finally, we note that $c^{\la}_{\mu,\eta(t)}=d^{\la}_{\mu,\eta(t)}$ for $t>0$ using the inductive hypothesis, since $\eta_2 < \nu_2$. Therefore, we also have that $\clmn=\dlmn.$

\end{proof}

\begin{prop} If $\la_2=\la_1$, then $\clmn$ is given by 
$$\la: \begin{bmatrix}
\fd_{333} & \fd^-_{322} & 0 \\
\mathfrak{p} & \fd_{231} & \\
\fd_{221} &&
\end{bmatrix} \qquad
\mu:\begin{bmatrix}
\fd^-_{212} & \mu_{12}\\
\fd_{221} & 
\end{bmatrix} \qquad
\nu:\begin{bmatrix}
\fd^+_{231} & \fd_{221} \\
\fd_{231} & 
\end{bmatrix}$$
\label{prop:la21}
\end{prop}

\begin{proof} Let $\dlmn$ be the expression given by the proposition. We use induction on $n=\mu_1+\nu_2-\la_1-s$ to show that $\clmn=\dlmn$. If $n=0$, then $\mu_1+\nu_2=\la_1-\mathfrak{p}=\la_2$. If $\mathfrak{p}=0$, this case reduces to that of Prop \ref{prop:nu1} (with the role of $\mu$ and $\nu$ switched). If $\mathfrak{p}=\la_3-\mu_2-\nu_3$, then $p=\la_1-\nu_1-\mu_3$, and so the condition $n=0$ implies that $\nu_1+\mu_3=\la_1$, which reduces this case to that of Prop \ref{prop:nu3} (once again with the role of $\mu$ and $\nu$ switched).

For $n>0$, let 
\begin{align*}
\ep&=(\nu_2-\nu_3),\\
\ze&=(\nu_1,\nu_3,\nu_3), 
\end{align*}
and 
\begin{align*}
\eta(t)&= \nu-(-t,t), \\
\ka(t)&=\la-(0,n-t,t).
\end{align*}
Note that $\eta(0)=\nu.$ 

By Lemma \ref{lem:path}, we have that 
$$\sum^n_{t=0} c^{\ka(t)}_{\mu,\ep} \cdot c^{\la}_{\ka(t),\ze} = \sum^n_{t=0} c^{\eta(t)}_{\ze,\ep}\cdot c^{\la}_{\mu,\eta(t)}.$$
We show that our hypothesized coefficients satisfy this equation.

We can determine $d^{\ka(t)}_{\mu,\ep}$ and $d^{\eta(t)}_{\ze,\ep}$ using the Pieri rule, $d^{\la}_{\ka(t),\ze}$ using Proposition \ref{prop:horn1}, and $d^{\la}_{\mu,\eta(t)}$ by the proposed formula for $\dlmn$. Using analogous calculations to those in the proof of Proposition \ref{prop:la23} and using the notation of Lemma \ref{lem:row}, one can check that:
\begin{align*}
d^{\la}_{\mu,\eta(t)}\cdot d^{\eta(t)}_{\ze,\ep}
&= d^{\la}_{\mu,\nu} \cdot d^{\nu}_{\ze,\ep} \cdot \frac{\phi^n_t(x;\si,\tau)}{\phi^n_0(x;\si,\tau)}, \\
d^{\ka(t)}_{\mu,\ep} \cdot d^{\la}_{\ka(t),\ze} 
&= d^{\la}_{\mu,\nu} \cdot d^{\nu}_{\ze,\ep} \cdot \frac{\phi^n_t(x;\tau,\si)}{\phi^n_0(x;\si,\tau)},
\end{align*}
where $\si=(2\mathfrak{p}+\fd_{323}+r,\mu_{13}-n+2r)$ and $\tau=(-\mathfrak{p}-\mu_{23}-r,0)$.

The result follows from Lemma \ref{lem:row}.
\end{proof}

\begin{prop} If $\la_1=\mu_2+\nu_2$, then $\clmn$ is given by
$$\la: \begin{bmatrix}
\fd_{111} & \fd^-_{223} & 0 \\
\fd^+_{222} & \fd_{113} & \\
\fd^+_{121} &&
\end{bmatrix} \qquad
\mu: \begin{bmatrix}
\fd^-_{112} & 0\\
\fd_{221} & 
\end{bmatrix} \qquad 
\nu: \begin{bmatrix}
\fd^+_{213} & \fd_{121} \\
\la^{+}_{23} & 
\end{bmatrix}$$
\label{prop:la1}
\end{prop}

\begin{proof} Let $\dlmn$ be the expression given by the proposition. We use induction on $n=\mu_1+\nu_2-\la_1-\mathfrak{p}$ to show that $\clmn=\dlmn$. Suppose $n=0$. Then either $\mu_1=\mu_2$ (if $\mathfrak{p}=0$), so that this case reduces to that of Prop. \ref{prop:horn2}, or $\la_2=\mu_3+\nu_1$ (if $\mathfrak{p}= \la_3-\mu_2-\nu_3$), which implies that $\la_3=\mu_1+\nu_3$, and so this case reduces to that of Prop \ref{prop:nu1} with the roles of $\mu$ and $\nu$ switched.

For $n>0$, let
\begin{align*}
\ep&=\om^\nu_1, \\
\ze&=\om^\nu_3+\om^\nu_2, 
\end{align*}
and
\begin{align*}
\eta(t)&= \nu-(-t,t), \\
\ka(t)&=\la-(0,n-t,n), 
\end{align*}
Note that $\eta(0)=\nu.$ 

By Lemma \ref{lem:path}, we have that 
$$\sum^n_{t=0} c^{\ka(t)}_{\mu,\ze} \cdot c^{\la}_{\ka(t),\ep} = \sum^n_{t=0} c^{\eta(t)}_{\ze,\ep}\cdot c^{\la}_{\mu,\eta(t)}.$$
We show that our hypothesized coefficients satisfy this equation.

We can determine  $d^{\la}_{\ka(t),\ep}$ and $d^{\eta(t)}_{\ze,\ep}$ using the Pieri rule, $d^{\ka(t)}_{\mu,\ze}$ using Proposition \ref{prop:horn2}, and $d^{\la}_{\mu,\eta(t)}$ for $t>0$ using the inductive hypothesis. Note that by the Pieri rule, $d^\nu_{\ze,\ep}=d^{\eta(t)}_{\ze,\ep}=1$. 

Thus, once again by an argument similar to that in the proof of Proposition \ref{prop:la23}, we have that:
\begin{align*}
d^{\la}_{\mu,\eta(t)}\cdot d^{\eta(t)}_{\ze,\ep}
&= d^{\la}_{\mu,\nu} \cdot \frac{\phi^n_t(x;\si,\tau)}{\phi^n_0(x;\si,\tau)} \mbox{ if }t>0, \\
d^{\ka(t)}_{\mu,\ze} \cdot d^{\la}_{\ka(t),\ep}
&= d^{\la}_{\mu,\nu} \cdot \frac{\phi^n_t(x;\tau,\si)}{\phi^n_0(x;\si,\tau)},
\end{align*}
where $\si=(0,n-\mu_{13}-2r)$ and $\tau=(\fd_{333}-\mathfrak{p}+r,\fd_{233}+2r)$.

The result follows from Lemma \ref{lem:row}.
\end{proof}

This completes our proof of Theorem \ref{thm:main}.

\section{Macdonald Polynomials} \label{sec:macd}

Jack polynomials $P_\la(\al;x)$ are generalized by Macdonald polynomials $P_\la(q,t;x)$, where

$$\lim_{t \rightarrow 1} P_\la(t^\al,t;x) = P_\la(\al;x).$$

Stanley's conjecture can be extended to Macdonald polynomials in a very straightforward way. Just as we had defined $\al$-generalizations of hook-length earlier, we can also define two $(q,t)$-generalizations:
\begin{itemize}
\item \emph{upper hook-length}: $h^*_\la(b) = 1-q^{a(b) +1}t^{\ell(b)}$
\item \emph{lower hook-length}: $h_*^\la(b) = 1-q^{a(b)}t^{\ell(b)+1}$
\end{itemize}

In fact, using these hook lengths, we can apply the Pieri rule (as stated in Theorem \ref{thm:jpieri}) to Macdonald polynomials as well (see \cite[IV.6.24]{M}). We can also get the appropriate analogue of Equation \ref{eq:transpose} by defining 
$$b_\la(q,t) = \frac{H_*^\la(q,t)}{H^*_\la(q,t)},$$
which gives us
$$\displaystyle c^{\la'}_{\mu',\nu'}\left(t,q\right) = \frac{\clmn(q,t) b_\mu(q,t) b_\nu(q,t)}{b_\la(q,t)}.$$

Finally, we use the definitions and theorem above to get the following extension of Theorem \ref{thm:main}.

\begin{thm} For a minimal triple $(\lmn)$ of partitions in $\Pn_3$, $\clmn(q,t)$ can be expressed by the same assignment of upper and lower hooks as $\clmn(\al)$, using the corresponding $(q,t)$-hooks instead of $\al$-hooks.
\label{thm:macmain}
\end{thm}

\begin{proof} We can still use the same classification as before, as well as the subsequent division number notation to give an assignment of upper and lower hooks. We can also once again express the ratio of a flipped to a standard $(q,t)$-hook in terms of $\phi$, in the following way. Given a $(q,t)$-hook $h_\la(b)$ (which could be either upper or lower), we define 
$$\hat{h}_\la(b)=\frac{h_\la(b)}{q^{a(b)}t^{\ell(b)}}.$$ 
Then 
\begin{align*}
\phi\left(t-q;\set{\hat{h}^*_\la(b)}\right) &= \frac{\hat{h}^*_\la(b)-(t-q)}{\hat{h}^*_\la(b)} \\ 
&= \frac{h^*_\la(b)-(t-q)(q^{a(b)}t^{\ell(b)})}{h^*_\la(b)} \\
&= \frac{h^\la_*(b)}{h^*_\la(b)},
\end{align*}
and 
\begin{align*}
\phi\left(t-q;\set{-\hat{h}^\la_*(b)}\right) &= \frac{-\hat{h}^\la_*(b)-(t-q)}{-\hat{h}^\la_*(b)} \\ 
&= \frac{h^\la_*(b)-(q-t)(q^{a(b)}t^{\ell(b)})}{h^\la_*(b)} \\
&= \frac{h^*_\la(b)}{h^\la_*(b)}.
\end{align*}
We can thus express our coefficients in terms of the same $\phi$ functions, but this time using the modified hook $\hat{h}(q,t)$ instead of the corresponding hook $h(\al)$, and using $x=t-q$ instead of $x=\al-1$. We can also prove that these expressions give the correct coefficient using the same lemmas as before with the same modification to $x$ and the hooks in $\si$ and $\tau$.

\end{proof}

To see that $\lim_{t\rightarrow 1} \clmn(t^\al,t) = \clmn(\al)$, we note that
\begin{align*}
\lim_{t\rightarrow 1} \frac{1-t^{\al(a(b))}t^{\ell(b)+1}}{1-t^{\al(a(b) +1)}t^{\ell(b)}} = \frac{\al(a(b))+\ell(b)+1}{\al(a(b) +1)+\ell(b)}, 
\end{align*}
as desired.

\section{Further Directions} \label{sec:future}
\subsection{Minimal triples in $\Pn_n$ with $n > 3$} 
The algebraic identities we obtain can be used in higher dimensions, but they relate to a system of minimal paths. As $n$ gets larger, the classification problem becomes much more complicated, both for unique LR fillings and for faces of Horn cones that correspond to minimal triples. Given the role played by the codimension one faces of Horn cones when $n=3$, one might wonder if Stanley's conjecture can be extended from minimal triples to all boundary triples on this cone. However, this is not true, as demonstrated by the example below. 

\begin{expl} Let $\la=(5,3,2,1),\mu=(3,2,1),\nu=(2,2,1)$. This lies on the codimension one face given by $\la_1=\mu_1+\nu_1$. However, in this case $\clmn(1)=2$ and 
$$\glmn(\al)=48\al^6(1+3\al)(3+5\al)(3+\al)(1+2\al)^2(3+2\al)(2+\al)^2(2\al^2+11\al+2).$$ 
\end{expl}

We also note that when $n=3$, our minimal paths typically involved decomposing $\mu$ or $\nu$ into rectangular blocks, since all triples involving a rectangle are minimal in this case. However, when $n>3$, it is possible to have a non-minimal triple even if $\nu$ is a rectangular partition. 

\begin{expl} Let $\la=(4,3,2,1),\mu=(3,2,1),\nu=(2,2)$. Then $\clmn(1)=2.$ 
\end{expl}

We therefore need a more efficient technique for determining minimal triples and for finding ways to expand them as minimal paths.

\subsection{Non-minimal triples in $\Pn_3$}
In general, if $\clmn(1) = k>1$, we cannot write it as the sum of $k$ different hook assignments for $\la,\mu,\nu$, each multiplied by a power of $\al$. Stanley and Hanlon demonstrated this in \cite{S} with the following example.

\begin{expl} $\la=(4,2,1),\mu=(3,1),\nu=(2,1)$
$$\glmn(\al) = 8\al^5(9 + 97\al + 294\al^2 + 321\al^3 + 131\al^4 + 12\al^5)$$
One can verify that any two experessions $f_1(\al)$ and $f_2(\al)$ given by hook assignments must share a common linear factor not equal to $\al$ or a common integer factor not equal to $8$. However, the above expression for $\glmn$ has no rational zeros besides $0$, and no integer factors besides $8$. 
\end{expl}

However, if we can expand the coefficients for non-minimal triples as a minimal path,
we might obtain a way to write coefficient as the sum of $k$ positive terms, each of which factors into linear factors in $\al$, given by $k$ terms of the form found in Lemmas \ref{lem:col} or \ref{lem:row}. In particular, we could write:
$$\clmn\cdot c^{\nu}_{\nu',\nu''} + \sum^n_{t=k+1} c^{\la}_{\mu,\eta(t)} \cdot c^{\eta(t)}_{\nu',\nu''}= \sum^n_{t=1} c^{\ka(t)}_{\mu,\nu'} \cdot c^{\la}_{\ka(t),\nu''},$$
and then write $\clmn$ in this form as long as the remaining coefficients are sufficient to find appropriate choices for $\si$ and $\tau$. This would give a combinatorial description of such coefficients, and also show that they are positive expressions in $\al$, as predicted by another conjecture of Stanley \cite[Conj. 8.3]{S}.

\section*{Acknowledgements}
I would like to thank my advisor Siddhartha Sahi for his patient and careful supervision of my dissertation research from which this work has been derived. I am also thankful to Anders Buch for many helpful discussions and suggestions for improvement.

\bibliography{jp}
\bibliographystyle{abbrv}    

\end{document}